\documentclass[11pt]{article}
    \title{\textbf{Intrinsic Finite Element Error Analysis on Manifolds with Regge Metrics, with Application to Calculating Connection Forms}}
    \author{Evan S. Gawlik\footnote{Santa Clara University,
500 El Camino Real,
Santa Clara, CA 95053, egawlik@scu.edu}, Jack McKee\footnote{Corresponding author, University of Hawaii at Manoa,
2565 McCarthy Mall (Keller Hall),
Honolulu, Hawaii 96822, jmckee@math.hawaii.edu}}
    
    \usepackage[bindingoffset=0.2in,
            left=1in,
            right=1.5in,
            top=1in,
            bottom=1in,
            footskip=.25in]{geometry}
    \usepackage{amsmath}
    \usepackage{amsfonts}
    \usepackage{amsthm}
    \usepackage{amssymb}
    \usepackage{xcolor}
    \usepackage{hyperref}
    \newtheorem{theorem}{Theorem}
    \newtheorem{definition}[theorem]{Definition}
    \newtheorem{corollary}[theorem]{Corollary}
    \newtheorem{lemma}[theorem]{Lemma}
    \newtheorem{remark}[theorem]{Remark}
    \newcommand{\pdiff}[2]{\frac{\partial #1}{\partial #2}}
    \newcommand{\iprod}[1]{\langle #1 \rangle}
    \newcommand{\correction}[2]{\textcolor{red}{}\textcolor{black}{#2}}
    \DeclareMathOperator*{\esssup}{ess\,sup}

\usepackage[style=numeric-comp,giveninits=true,url=false,maxbibnames=100,backend=biber,doi=false,isbn=false,url=false,eprint=false,date=year]{biblatex}
\usepackage{graphicx}
\begin{document}

\maketitle 

\begin{abstract}
    We present some aspects of the theory of finite element exterior calculus as applied to partial differential equations on manifolds, especially manifolds endowed with an approximate metric called a Regge metric. Our treatment is intrinsic, avoiding wherever possible the use of preferred coordinates or a preferred embedding into an ambient space, which presents some challenges but also conceptual and possibly computational advantages. As an application, we analyze and implement a method for computing an approximate Levi-Civita connection form for a two-dimensional Riemannian manifold with trivial homology.
\end{abstract}

\footnotetext{Keywords: Finite Elements, FEEC, Manifolds, Regge Metric, Discrete Differential Geometry}

\footnotetext{2020 Mathematics Subject Classification: Primary 65N30, 58J32
}

\section{Introduction}

The analysis of finite element methods on manifolds is a fascinating and quite difficult problem. It is key to producing efficient algorithms for solving PDEs which naturally lie on a surface or which have a natural metric associated to them, such as small vibrations of a curved surface \cite{Glinskii}, diffusion on a curved surface \cite{Faraudo}, and thin-film or shallow-water approximations of the Navier-Stokes equations on curved spaces (such as the surface of the Earth) \cite{Samavaki}.
Numerical discretizations of such PDEs that automatically respect key topological, analytic, and geometric aspects of their domain, such as those appearing in the Finite Element Exterior Calculus (FEEC) framework \cite{Arnold,Arnold-Winther}, are highly desirable due to their wide applicability. 
Although robust convergence theorems have been developed for various types of discretizations of PDEs on Riemannian manifolds and approximations of Riemannian manifolds \cite{Dziuk_Elliott_2013,Holst-Stern,demlow2009higher}, typically these theorems have been directed toward embedded hypersurfaces in $\mathbb{R}^n$, where the metric is inherited from Euclidean space. (A notable exception is~\cite{christiansen2007stability}, where compact manifolds are treated without reference to an embedding.)  However, this restricts the class of manifolds which may be treated, in particular it excludes $(n-1)$-dimensional manifolds which do not possess a natural embedding (or any embedding at all) into $\mathbb{R}^n$.

Recently, Martin Licht has described commuting cochain projections for FEEC on manifolds \cite{Licht}, building upon the foundational work of Holst and Stern \cite{Holst-Stern} and providing the first (to our knowledge) complete proof of the convergence of a finite element method for the Hodge-Laplace problem on manifolds. Our work mostly runs in parallel, with the notable exception that we take the existence of commuting cochain projections as given. However, our treatment of mesh quality measures as well as the setup of the computational problem diverge. We treat the true metric and the approximate metric as existing on the same smooth manifold $M$ and perform all the analysis intrinsically. The approximate metric is a \emph{Regge metric}, that is, a piecewise-smooth metric with some continuity properties. Much of this paper is devoted to building and using machinery for finite elements on manifolds with Regge metrics.

Our initial motivation for studying intrinsic approximations of the Hodge-Laplace problem stemmed from trying to solve another problem that, at the outset, seemed much simpler than it was. 
On a 2-dimensional Riemannian manifold $M$ which has trivial homology, there is a unique frame such that the corresponding connection form is co-exact. This means it is possible to pose the problem of finding the corresponding connection form as a Hodge-Laplace problem on $M$ with an appropriate source term and boundary conditions. Most interestingly the form, but not its corresponding frame, can be computed only knowing the Gaussian curvature of $M$ and the angle that the frame makes with the normal vector on $\partial M$.  We show how to do this in Section 4. Then we prove convergence of a corresponding numerical scheme in Section 5.

Regge metrics, especially their associated quantities such as curvature and connections, have been studied extensively~\cite{cheeger1984curvature,christiansen2024definition,gawlik2020high,BKGa22,gopalakrishnan2022analysis,GaNe22,gawlik2023finite,gopalakrishnan2023analysis,gopalakrishnan2024improved}, often in the context of numerical general relativity and nonlinear elasticity \cite{Regge,christiansen2011linearization,Li,neunteufel2021avoiding}. We are unaware of any existing work using a Regge metric to approximate the Hodge-Laplace equation with a metric-dependent source term as we have done\correction{ (although Regge calculus has been used for discrete approximations of the Laplacian in the physics literature \cite{Calcagni_2013})}{}, however it is a relatively simple matter to specialize the abstract theory of geometric variational crimes once the necessary background is established. The geometric error incurred depends only on quantities intrinsic to $M$, and local coordinates may be used to take advantage of existing finite element spaces and their approximation properties. We take this up in Section 3.

One key thing that distinguishes our analysis from prior work is our adherence to intrinsic calculations. We take the point of view that all convergence results for a system posed in a manifold should be stated in terms of quantities intrinsic to the manifold, wherever possible avoiding the use of coordinate-dependent expressions. This means that much of the basic theory, as well as scaling arguments, is much more involved than it otherwise would be. It is our belief, however, that these are necessary steps to decouple as much as possible the description of the computational mesh from the geometry of the problem. In Section 2 we state and/or prove some elementary properties of the $L^2$, $H(d)$ (meaning those $L^2$ differential forms which have an $L^2$ exterior derivative), and $H^1$ spaces of differential forms on manifolds with Regge metric and the trace and inverse inequalities for curved elements, as well as fairly ad-hoc, though in our context quite natural, definitions of shape regularity and quasi-uniformity of simplicial meshes. Our bounds do not depend on any embedding of the manifold into an ambient space or any coordinate-dependent expressions, although we do also provide weaker bounds in terms of coordinate-dependent expressions.

\paragraph{Computing connection forms.}

As we alluded to above, finding/justifying the appropriate discretization of the connection form problem was a deceptively involved process. The difficulty stems from the source term and boundary conditions, which must be treated with great care to produce a well-posed and convergent numerical discretization. Our efforts spurred us to write down a generalized Gauss-Bonnet Theorem for manifolds with Regge metrics in a recent preprint~\cite{gawlik2025curvaturereggemetrics}.  With the appropriate source term and boundary conditions, the convergence analysis of our numerical scheme can be carried out with the help of existing work on the distributional scalar curvature of manifolds with Regge metrics. In Section 5 we rigorously prove the convergence of the computed connection form in our intrinsic setting. We briefly demonstrate numerical convergence for a simple model problem in Section 6.

It is natural to wonder whether the computed connection form has any practical uses.  In fact it does; our preliminary work indicates that one can leverage it to compute the spectrum of the Bochner Laplacian.  (Recall that the Bochner Laplacian, which differs from the Hodge Laplacian, is constructed by composing the covariant derivative with its $L^2$-adjoint.)  We plan to pursue this in future work.

\paragraph{Contributions.}
In Section 2 we prove (to our knowledge) the first scaled trace and inverse inequalities applied to arbitrary curved finite elements which make no reference to an ambient space or special coordinates. We also arrive at new definitions of shape regularity and quasi-uniformity for manifolds with Regge metrics, and we prove that Euclidean shape regularity and quasi-uniformity in coordinates implies Riemannian shape-regularity and quasi-uniformity\correction{}{, thereby showing that the new definitions are generalizations of the standard definitions}. In Section 3, \correction{we specialize FEEC methods to the situation where a Regge metric is used to approximate a smooth Riemannian metric}{we apply existing results from the theory of geometric variational crimes~\cite{Holst-Stern} to a setting where a Regge metric approximates a smooth Riemannian metric}, and we provide \correction{}{new} intrinsic bounds on the geometric error. In Sections 4, 5, and 6 we apply these methods to derive and prove convergence of an approximate connection form for compact 2-dimensional Riemannian manifolds with trivial homology. We have attempted to point out where our results \correction{}{in these latter sections} are based on prior work, especially Theorem \ref{diffFh} is an extension of an existing proof by the first author and collaborators, and Theorem \ref{Fhbounds} is inspired by that proof as well. 

\paragraph{Notation.}
This section will serve as a guide for the notation used in this paper.

The language of differential forms and exterior calculus is used heavily, as is the Einstein summation convention. The reader is referred to \cite{Darling}, \cite{Flanders}, \cite{Marsden}, and \cite{deRham} for background. All manifolds under consideration are allowed to be manifolds with corners (per definitions used in \cite{Knapp})\correction{}{, a class which includes simple polytopes in $\mathbb{R}^n$}. We use $T_x(M)$ or $T_xM$ to represent the tangent space of a manifold $M$ at the point $x \in M$, $T^*_x(M)$ or $T^*_xM$ to represent the cotangent space at $x \in M$, and $\Lambda^k_x(M) := \wedge_{i = 1}^k T^*_xM$ is the space of alternating $k$-covectors at $x \in M$. Without the subscript $x$ these refer to the tangent bundle, cotangent bundle, and bundle of alternating covariant $k$-covectors respectively. Often the manifold $M$ in consideration is called $T$, leading to expressions like $T(T)$. $C^\infty\Omega^k(M)$ refers to the space of smooth sections $M \to \Lambda^k(M)$, i.e.,~the space of differential forms. For a smooth diffeomorphism $f: M \to N$, we make frequent use of the notation $f^{-*} := (f^{-1})^*$, where $f^*: C^\infty\Omega^k(N) \to C^\infty\Omega^k(M)$ is the pullback map.  If $M$ is a submanifold of $N$, then the inclusion map $M \hookrightarrow N$ is denoted $i_M$.

Multi-index notation is also used heavily. A multi-index is a tuple $I = (i_1,i_2,\dots,i_k)$, where $1 \le i_j < i_{j+1}$ for each $j$ and $i_k \le n$, where $n$ is (usually) known from context as the dimension. The length of a multi-index is denoted $|I|$ and is equal to $k$. The special multi-index $[k]$ is simply $(1,2,\dots,k)$.  The symmetric group on $\{1,2,\dots,k\}$ is denoted $S_k$.

Submanifolds $T$, $e$, and $p$ will be implicitly assumed to refer to $n$-dimensional simplices, $(n-1)$-dimensional simplices, and points (respectively) in a smooth triangulation of an $n$-dimensional manifold. All such objects are considered to be closed unless stated otherwise. In sections 4, 5, and 6, because $n=2$, these will refer to triangles, edges, and points respectively.

An integral over the boundary of a manifold with corners is always implicitly understood to mean the sum of integrals over each codimension-1 smooth segment of the boundary (in other words, the stratum of index 1 \cite{Knapp}).

A subscript $t$ will mean that a quantity is dependent on a metric $g(t)$. A subscript $h$ will mean that a quantity is dependent on a metric $g_h$. Occasionally we use a subscript $g$ to emphasize that a quantity is dependent on the metric $g$. \correction{}{In some cases, when the implied metric is clear from the context, we suppress subscripts altogether.} All metrics that we consider are Regge metrics, which in our case means they are smooth Riemannian metrics on each top-dimensional simplex, and the tangential-tangential components of the metric tensor are single-valued on any shared face of two simplices of any dimension. In particular, every smooth metric is a Regge metric.

Frequently we will take norms of matrices of 1-forms/matrix-valued 1-forms. The notation $\|\cdot\|_{\_,g}$ will mean the norm of type $\_$ (e.g., the matrix $2$-norm $\|\cdot\|_2$, Frobenius norm $\| \cdot\|_F$, $\infty$-norm $\|\cdot\|_\infty$, or maximum-entry norm $\|\cdot\|_m$) maximized over the set of all norm-$1$ input vectors $V$, as measured by $g$. Specifically, $\|A|_x\|_{\_,g} = \sup_{V \in T_x(M), V \ne 0}\frac{\|A(V)\|_{\_}}{\|V\|_g}$.

A quantity in double brackets $[\![X]\!]$ has double meaning: when the quantity is evaluated along a facet shared by two simplices, $[\![X]\!] := X^+ - X^-$, where $X^+$ is evaluated using the metric on one side and $X^-$ is evaluated using the metric on the other side, both with the corresponding induced orientations on the boundaries of each triangle. In the case that the quantity is evaluated along a positively oriented edge that lies in the boundary of only one triangle, $[\![X]\!] := X^+$. In the case that the quantity is evaluated at a corner $p$ of a triangle $T$, $[\![X|_T]\!] := X^+ - X^-$ where $X^+$ is the quantity evaluated on the ``counterclockwise" side and $X^-$ is the quantity evaluated on the ``clockwise" side intersecting at $p$, with the order given by the  orientation on $T$. In one instance, the jump notation is used to refer to the jump of a scalar quantity across a corner on the boundary in the same way.

For two normed spaces $X$ and $Y$, $\|\cdot\|_{\mathcal{L}(X,Y)}$ is the operator norm on linear maps $X \to Y$.

\section{Sobolev Spaces on Manifolds with Regge Metrics}
\subsection{Regge Metrics and Their $L^2$, $H(d)$, and $H^1$ Spaces}
A \emph{manifold with Regge metric} $(M,\mathcal{T},g)$ is a smooth oriented $n$-dimensional manifold $M$ (possibly noncompact and possibly with boundary or corners), with a smooth countable triangulation $\mathcal{T}$, and a Regge metric $g$ for this triangulation. 

By a smooth countable triangulation of $M$ we mean that $\mathcal{T}$ is a countable cover of $M$ by closed sets $\Delta$, each of which is the image of a smooth embedding from the standard $k$-simplex $\hat\Delta^k = \{x \in \mathbb{R}^k: \sum_{i = 1}^k x^i \le 1, 0 \le x^i \le 1 \forall i\}$ for some $0 \le k \le n$, each face of each simplex is also in $\mathcal{T}$, and the intersection of any two simplices in $\mathcal{T}$ is either empty or a face of both simplices. If the relative interior of a simplex intersects a stratum of $\partial M$, meaning it contains a corner point of any degree, it must lie fully within that stratum. Note that we are abusing terminology slightly by referring to the smooth image of a simplex, which may be curved, as a simplex. 

By a Regge metric $g$ for $\mathcal{T}$, we mean the data of a smooth Riemannian metric $g_{\Delta}$ for each $\Delta \in \mathcal{T}$, with the compatibility condition $i_f^* g_{\Delta} = g_{f}$ for each face $f \subset \Delta$. In other words, the tangential-tangential components of $g$ are single-valued on any shared face of two simplices of any dimension. In particular, the tangential-tangential components of $g$ must be continuous across an interface $e = T \cap T'$, when $T$ and $T'$ are $n$-dimensional simplices, while the normal-tangential and \correction{tangential-tangential}{normal-normal} components may be discontinuous. Frequently, the triangulation will be implicitly included with the data of the Regge metric $g$, and we just write $(M,g)$.

For now, no further assumptions will be put on $(M,\mathcal{T},g)$. In this section we collect some properties of spaces of differential forms on manifolds with Regge metrics. The development is quite similar to those in \cite{Licht,Schwarz}, but adapted for Regge metrics. We will therefore omit some details for brevity.

We begin by describing the inner product induced by the metric $g_T$ on $\Lambda^k_x(T)$ for every $T \subseteq M, x \in T, k \in \mathbb{N}$. This inner product in turn is induced from the inner product on $T^*_x(T)$:
\[
\iprod{\alpha,\beta} := \alpha(\beta^\sharp),
\]
where $\beta^\sharp$ is the unique element of $T_x(T)$ such that, for all $V \in T_x(T)$, $g_T(\beta^\sharp,V) = \beta(V)$. The inner product on $\Lambda^k_x(T)$ is then defined by
\[
\iprod{\alpha^1\wedge\dots\wedge\alpha^k,\beta^1\wedge\dots\wedge\beta^k} := \sum_{\sigma \in S_k} \mathrm{sign}(\sigma)\prod_{i = 1}^k \iprod{\alpha^i,\beta^{\sigma(i)}}
\]
and extending multilinearly. This inner product induces the usual corresponding norm $\|\alpha\| := \sqrt{\iprod{\alpha,\alpha}}$.
 Notably, the inner product of covectors is dual to the contraction of vector fields: $\alpha(X_1,\dots,X_k) = \iprod{\alpha,X_1^\flat\wedge \dots \wedge X_k^\flat}$, where $\flat$ denotes the inverse of $\sharp$, so the induced norm is the same as the operator norm on $(\displaystyle\bigwedge^kT_x(T))^* \cong \Lambda^k_x(T)$. 
 
Another, often more convenient, way to write the inner product of two alternating $k$-covectors is to use the Hodge star, which is defined in terms of the aforementioned inner product. We set $\star \alpha$ to be the $(n-k)$-covector such that $\iprod{\beta,\alpha}dV = \beta\wedge\star\alpha$ for all $\beta \in \Lambda^k_x(T)$, where $dV$ is the volume form induced by $g$. This map is clearly an invertible isometry, and its inverse is in fact $(-1)^{k(n-k)}\star$.

Lastly we note that the natural tensor product metric $\iprod{\alpha \otimes \beta,\gamma \otimes \delta} := \iprod{\alpha,\gamma}\iprod{\beta,\delta}$ can be used to extend this inner product to $T_x^*(T)^{\otimes p} \otimes \Lambda^k_x(T)$. We will use the shorthand $\Lambda^{p,k}_x(T) := T_x^*(T)^{\otimes p}\otimes \Lambda^k_x(T)$.  Similarly, $\Lambda^{p,k}(T) := T^*(T)^{\otimes p}\otimes \Lambda^k(T)$.  The space of smooth sections $T \to \Lambda^{p,k}(T)$ will be denoted $C^\infty\Omega^{p,k}(T)$, and the subspace of compactly supported smooth sections $T \to \Lambda^{p,k}(T)$ will be denoted $C_c^\infty\Omega^{p,k}(T)$.

A set $X \subseteq M$ is \emph{$\Sigma$-measurable} if $\phi(X \cap U)$ is Lebesgue measurable for each coordinate chart $(U,\phi)$ of $M$. The condition of being measurable is preserved by smooth coordinate changes, so this is a well-defined $\sigma$-algebra $\Sigma$ on $M$, and it is additionally a Borel $\sigma$-algebra. A \emph{$\Sigma$-measurable $k$-form} is a section $M \to \Lambda^k(M)$ whose coefficients in each coordinate chart are measurable functions $\mathbb{R}^n \to \mathbb{R} \cup \{\pm \infty\}$. Again this condition is invariant under smooth coordinate changes, so it makes sense. 

A $\Sigma$-measurable $k$-form $\alpha$ is \emph{locally integrable} if $\sum_{T \subseteq M} \int_T\|\alpha \wedge \phi\|dV < \infty$ for all smooth compactly supported test forms $\phi \in C^\infty_c\Omega^{n-k}(M)$, or equivalently, if each coefficient of $\alpha|_T$ is locally integrable for any simplex $T \subset M$, and for any coordinates on $T$. The set of locally integrable $k$-forms will be called $L^1_{\mathrm{loc}}\Omega^k(M)$, and they enjoy the property that $\int_M \alpha \wedge \phi$ is well-defined for any of the test forms. Note that which metric is used to test local integrability does not matter, as $\|\alpha \wedge \phi\|_{g_1}dV_{g_1} = \|\alpha \wedge \phi\|_{g_2}dV_{g_2}$ a.e. for any $g_1, g_2$. Likewise the integral $\int_T \alpha \wedge \phi$ on each $n$-simplex $T \subseteq M$ is defined intrinsically and is independent of any particular metric, because $\alpha \wedge \phi$ is a multiple of the smooth volume form by a Lebesgue integrable function when pulled back along the embedding map $\hat\Delta^n \tilde\rightarrow T \subseteq M$.

\begin{remark}
    When one considers open, bounded domains in $\mathbb{R}^n$, compact support implies vanishing on the boundary. However in our case, $M$ can be compact, or can be noncompact but include its boundary. Compact support of a smooth form in this case is used to guarantee a well-defined, finite integral but does not imply anything about boundary conditions.  We will focus on compact manifolds with corners rather than open domains in $\mathbb{R}^n$, because open domains cannot support a finite triangulation, and the Hodge theory for noncompact, incomplete manifolds is substantially more complicated \cite{Junya}. Focusing in this way means that some definitions below appear slightly nonstandard; for instance, the set of compactly supported smooth functions is dense in the space $H(d)\Omega^0(M,g)$ defined below.
\end{remark}

If $\alpha,\beta$ are locally integrable $k$-forms, then in particular they are locally integrable when restricted to each simplex $T \subseteq M$. Because the Regge metric $g$ restricts to the smooth metric $g_T$ on $T$, the usual theory of Sobolev spaces of functions carries over almost unchanged to the space of tensor-valued differential forms on $T$ \cite[p. 30]{Schwarz}. The space $L^2\Omega^{p,k}(T,g)$ is the completion of $C^\infty_c \Omega^{p,k}(T)$ under the norm induced by the inner product
\[\iprod{\alpha,\beta}_{L^2(T,g)} := \int_{\mathring{T}} \iprod{\alpha,\beta}dV\correction{.}{,}\]
\correction{}{where the integrand on the right is implicitly metric-dependent.}
Of particular interest at this level is the operator $\nabla: C^\infty\Omega^{p,k}(T) \to C^\infty\Omega^{p+1,k}(T)$. For a form $\alpha \in C^\infty\Omega^{p,k}(T)$ and a vector field $Y$ on $T$, the covariant derivative $\nabla_Y \alpha$ is defined by
\[
(\nabla_Y \alpha)(X_1,\dots,X_{p+k}) = Y(\alpha(X_1,\dots,X_{p+k})) - \sum_j \alpha(X_1,\dots,\nabla_YX_j,\dots,X_{p+k}).
\]
It is easily verified that $\nabla_Y \alpha$ is a $T^*(T)^{\otimes p}$-valued differential $k$-form and $\nabla_{fY}\alpha = f\nabla_Y \alpha$. Therefore $\nabla \alpha$ is a smooth section of $T^*(T)^{\otimes p+1} \otimes \Lambda^k(T)$. 

This operator is by definition bounded on the space $H^1\Omega^{p,k}(T,g)$, which is the completion of $C^\infty\Omega^{p,k}(T)$ with the norm 
\[\|\alpha\|_{H^1(T,g)} := \sqrt{\|\alpha\|_{L^2(T,g)}^2 + \|\nabla \alpha\|_{L^2(T,g)}}.\]

With the $L^2(T,g)$ spaces and inner products defined for individual top-dimensional simplices $T$, it is possible to define the $L^2(M,g)$ inner product as simply the sum over all top-dimensional simplices in the triangulation, i.e., $\iprod{\alpha,\beta}_{L^2(M,g)} := \sum_{T \subseteq M} \iprod{\alpha,\beta}_{L^2(T,g)}$. The space $L^2\Omega^{p,k}(M,g)$ is defined as the completion of $C^\infty_c\Omega^{p,k}(M)$ under the induced norm.

It is tempting to define the $H^1$ inner product on $M$ in a similar way, and such a product is useful, but we avoid calling it ``the $H^1$ inner product'' on $M$, because due to discontinuities of the metric $g$, a form $\alpha$ such that $\nabla\alpha|_T \in L^2\Omega^{1,k}(T,g)$ for each $T \subseteq M$ does not necessarily have a weak covariant derivative that is in $L^2\Omega^{1,k}(M,g)$. Properly defining weak covariant derivatives is outside of the scope of this article.

The exterior differential operator $d$ and its formal dual $\delta = (-1)^{n(k-1) - 1}\star d\star$ can be defined weakly for locally integrable forms, but only on certain test forms, namely those who have no tangential (respectively, normal) component on the boundary. Formally, we define
\[
C^\infty_c\Omega^k_\mathbf{t}(M) := \{\phi \in C^\infty_c\Omega^k(M) : i_{\partial M}^*\phi = 0\}.
\]

Smooth tangential-free forms are dense in $L^2\Omega^k(M,g)$. To show this, we just need to show that any form $\alpha \in C^\infty_c\Omega^k(M)$ can be approximated by forms in $C^\infty_c\Omega^k_{\mathbf{t}}(M)$. If $\rho_\epsilon$ is a smooth compactly supported bump function which is equal to $1$ on $\partial M \cap \mathrm{supp}(\alpha)$, is less than or equal to $1$ on $\mathring{M}$, and such that the support of $\rho_\epsilon$ has volume $\epsilon$, then $(1 - \rho_\epsilon)\alpha \in C^\infty_c\Omega^{p,k}_{\mathbf{t}}(M)$ and $\|(1 - \rho_\epsilon)\alpha - \alpha\|_{L^2(M,g)} = \|\rho_\epsilon\alpha\|_{L^2(M,g)} \le \epsilon^{1/2} \sup_{x \in M} \|\alpha|_x\|_g$, which converges to zero as $\epsilon \to 0$. Such bump functions can be constructed because $\partial M \cap\mathrm{supp}(\alpha)$ is compact.

If $\alpha \in L^1_\mathrm{loc}\Omega^k(M)$, then the distributional exterior derivative of $\alpha$ will be defined by
\[
\forall \phi \in C^\infty_c\Omega^{n-k-1}_\mathbf{t}(M), \;\;\iprod{\iprod{d\alpha,\phi}} := (-1)^{k+1}\int_M \alpha \wedge d\phi.
\]
When $\alpha$ is a smooth form, $\iprod{\iprod{d\alpha,\phi}} = \int_M d\alpha \wedge \phi$.
We define the $H(d)$ norm for forms $\alpha \in L^2\Omega^k(M,g)$ by
\[
\|\alpha\|_{H(d)(M,g)} := \|\alpha\|_{L^2(M,g)} + \sup_{\phi \in C^\infty_c\Omega^{n-k-1}_\mathbf{t}(M) \backslash \{0\}} \frac{\iprod{\iprod{d\alpha,\phi}}}{\|\phi\|_{L^2(M,g)}}.
\]

The space $H(d)\Omega^k(M,g) \subset L^2\Omega^k(M,g)$ is defined as the closure of $C^\infty_c\Omega^k(M)$ with the topology induced by the $H(d)$ norm. Note that if $\alpha \in H(d)\Omega^k(M,g)$, then $d\alpha \in L^2\Omega^{k+1}(M,g)$, because $d\alpha$ as a functional is bounded in the $L^2$ norm on the dense subset $C^\infty_c\Omega^{n-k-1}_{\mathbf{t}}(M)$.

The map $d: H(d)\Omega^k(M,g) \to L^2\Omega^{k+1}(M,g)$ is continuous by definition and agrees with its smooth counterpart when evaluated on smooth forms. $H(d)\Omega^k(M,g)$ therefore has a natural Hilbert space structure $\iprod{\alpha,\beta}_{H(d)(M,g)} := \iprod{\alpha,\beta}_{L^2(M,g)} + \iprod{d\alpha,d\beta}_{L^2(M,g)}$. We will see later that $H(d)\Omega^k(M,g)$, as a topological space, depends on the metric only up to quasi-isometry, meaning the ``small-scale'' differences in geometry are inconsequential.  (See Section~\ref{sec:bounds} for the definition of quasi-isometry.)

With the distributional exterior derivative $d$ and the $H(d)$ space defined, the distributional codifferential $\delta$ and $H(\delta)$ can be defined fairly easily, with $\delta := (-1)^{n(k-1) - 1} \star d \star$ and $H(\delta)\Omega^k(M,g) := \star^{-1}(H(d)\Omega^{n-k}(M,g))$. The following formula is just integration by parts for smooth forms, and the left-hand side remains well-defined if $\alpha \in  H(d)\Omega^k(M,g)$ and $\beta \in H(d)\Omega^{n-k-1}(M,g)$:
\begin{equation}\label{dist_exterior_deriv}
\int_M d\alpha \wedge \beta + (-1)^k \int_M \alpha \wedge d\beta = \int_{\partial M} i_{\partial M}^*(\alpha \wedge \beta).
\end{equation}

So, while $i_{\partial M}^* \alpha$ is not necessarily well-defined, there can still be a notion of ``trace zero" if $k < n$. When (\ref{dist_exterior_deriv}) is zero for all $\beta$ we will say that $\alpha \in H(d)\Omega_\mathbf{t}^k(M,g)$. We will also set $H(\delta)\Omega^k_\mathbf{n}(M,g) := \star^{-1}(H(d)\Omega_\mathbf{t}^{n-k}(M,g))$. As one would expect, as long as $\alpha \in H(d)\Omega^k(M,g)$ and $\beta \in H(\delta)\Omega^{k+1}(M,g)$, we have
\[
\iprod{d\alpha,\beta}_{L^2(M,g)} = \iprod{\alpha,\delta\beta}_{L^2(M,g)}
\]
if $\alpha \in H(d)\Omega^k_\mathbf{t}(M,g)$ or $\beta \in H(\delta)\Omega^{k+1}_\mathbf{n}(M,g)$. In other words, $H(\delta)\Omega^k_{\mathbf{n}}(M,g)$ is the $L^2$ dual of $H(d)\Omega^k(M,g)$ and $H(d)\Omega^k_{\mathbf{t}}(M,g)$ is the dual of $H(\delta)\Omega^k(M,g)$.

The spaces $H(d)\Omega^k(M,g)$ and $H(\delta)\Omega^k(M,g)$ contain subspaces consisting of compactly supported piecewise-smooth forms. It is useful to understand what these subspaces look like.

If $i_e^*[\![\alpha]\!] \ne 0$ for some $(n-1)$-simplex $e$ that is not contained in $\partial M$, that is if the tangential components of $\alpha$ are not continuous, then test forms $\phi$ can be constructed so that $\|\phi\|_{L^2(M,g)}$ is arbitrarily small but $\iprod{\iprod{d\alpha,\phi}}$ is arbitrarily large, which would mean that $\|\iprod{\iprod{d\alpha,\cdot}}\|_{L^2\Omega^{n-k-1}(M,g)^*} = \infty$. Therefore the compactly supported piecewise-smooth elements of $H(d)\Omega^k(M,g)$ are those which are continuous in their tangential components across $n$-simplex boundaries. This gives some intuition as to why the space $H(d)\Omega^k(M,g)$, and the map $d$, are independent of the small-scale geometry of $g$, since this is a purely topological condition.

Similarly, the compactly supported piecewise-smooth elements of $H(\delta)\Omega^k(M,g)$ are those which are continuous in their normal components, as defined by $g$. These will generally be discontinuous and will depend on the small-scale geometry of $g$, meaning they are significantly less convenient to work with when designing a finite element method.

\subsection{Bounds on Sobolev Norms from Differing Metrics} \label{sec:bounds}

A measurement of the ``difference'' between two Regge metrics is an important ingredient for analyzing approximation properties of metrics and PDE problems depending on them. The notion of \emph{quasi-isometry}, fundamental in geometric group theory \cite{Gromov} and long studied in relation to approximations (both analytical and metrical) of Riemannian manifolds \cite{Kanai}, turns out to provide a natural and powerful tool to this end. 

Given two Regge metrics $(\mathcal{T}_1,g_1)$ and $(\mathcal{T}_2,g_2)$ on the same manifold $M$, and a smooth submanifold $\Omega \subseteq M$, we define
\begin{equation}
C_{g_1,g_2}(\Omega) := \esssup_{x \in \Omega, \; v \in T_x\Omega \backslash \{0\}} \frac{\|v\|_{g_1}}{\|v\|_{g_2}}
\end{equation}
and
\begin{equation}
D_{g_1,g_2}(\Omega) := \esssup_{x \in \Omega} \|dV_{g_1}|_x\|_{g_2},
\end{equation}
where $dV_{g_1}$ is understood to mean the $g_1$-volume form of $\Omega$, and a set $X$ has measure 0  if $\mu_\mathrm{Lebesgue}(\phi(X \cap U)) = 0$ for any coordinate chart $(U,\phi)$ of $\Omega$. We define $C_{g_1,g_2}(\emptyset)$ and $D_{g_1,g_2}(\emptyset)$ to be zero for convenience, as these quantities are always nonnegative.

When $\Omega = \bigcup_i \Omega_i$ is a finite union of smooth submanifolds (such as the boundary of a simplex), $C_{g_1,g_2}(\Omega) := \sup_i C_{g_1,g_2}(\Omega_i)$ and likewise for $D_{g_1,g_2}(\Omega)$.

If $0 < C_{g_1,g_2}(M) < \infty$ and $0 < C_{g_2,g_1}(M) < \infty$, we will say that $(M,\mathcal{T}_1,g_1)$ and $(M,\mathcal{T}_2,g_2)$ are \emph{quasi-isometric}. This induces an equivalence relation on the set of smooth manifolds with Regge metrics. Quasi-isometry can be interpreted to mean that two manifolds are ``asymptotically the same''. It is very straightforward to show that if $\mathcal{T}_1$ and $\mathcal{T}_2$ are both finite triangulations, then $g_1$ and $g_2$ are quasi-isometric, and hence a compact manifold $M$ supports only one quasi-isometry class.

As one might expect, $L^2$ spaces of forms depend only on the quasi-isometry class of the metric.

\begin{theorem}\label{L2_equiv}
If $(\mathcal{T}_1,g_1)$ and $(\mathcal{T}_2,g_2)$ are two Regge metrics on the same manifold $M$, then
\begin{enumerate}
\item{$ \frac{1}{C_{g_1,g_2}(M)^n} \le D_{g_2,g_1}(M) \le C_{g_2,g_1}(M)^n$},
\item{$\displaystyle{\forall \alpha \in L^2\Omega^{p,k}(M,g_2), \;\|\alpha\|_{L^2(M,g_1)}^2 \le{n \choose k}n^pC_{g_2,g_1}(M)^{2p + 2k}D_{g_1,g_2}(M)\|\alpha\|_{L^2(M,g_2)}^2}$.}
\end{enumerate}
\end{theorem}

\correction{}{A proof of this theorem can be found in the appendix.}

This theorem will be useful for proving the scaling properties for the trace and inverse inequalities. It could be tightened substantially by using more information about the metrics, but $C_{g_1,g_2}(M)$ and $D_{g_1,g_2}(M)$ are easy to compute and widely applicable. It also has an immediate corollary:
\begin{corollary}
    If $g_1$ and $g_2$ are two quasi-isometric Regge metrics on $M$, then, for all $p,k$, $L^2\Omega^{p,k}(M,g_1) = L^2\Omega^{p,k}(M,g_2)$ and $H(d)\Omega^k(M,g_1) = H(d)\Omega^k(M,g_2)$ as topological vector spaces.
\end{corollary}

In what follows we will frequently need to make use of comparisons between the covariant derivatives of forms. The following lemma establishes a relationship between covariant derivatives in different metrics, which is key to proving the trace and inverse inequalities.

\begin{lemma}
\label{compare_cov}
Given two Regge metrics $g$ and $g'$ on a manifold $M$ and their Levi-Civita connections $\nabla$ and $\nabla'$, $\nabla - \nabla'$ is a bundle morphism $T^*(T \cap T')^p \otimes \Lambda^k(T \cap T') \to T^*(T \cap T')^{p+1} \otimes \Lambda^k(T \cap T')$ on each $T, T' \subseteq M$ where $g$ and $g'$ are smooth respectively. Furthermore, if $U \subseteq T \cap T'$ is measurable and contained in a coordinate chart with coordinates $\{x^i\}_{i = 1}^n$, then
\begin{align*}
\|\nabla - \nabla'\|_{\mathcal{L}(L^2\Omega^{p,k}(U,g),L^2\Omega^{p+1,k}(U,g))} \le \frac{3(k+p)n^\frac{7}{2}}{2}&\|G^{-1}\|_{L^\infty(U)}\|G\|_{L^\infty(U)}\\
&\cdot \big(\|G^{-1} - G'^{-1}\|_{L^\infty(U)}\|dG\|_{L^\infty(U,g)} \\
&\quad\quad + \|G'^{-1}\|_{L^\infty(U)}\|dG - dG'\|_{L^\infty(U,g)}\big),
\end{align*}
where $G_{ij} = \iprod{\pdiff{}{x^i},\pdiff{}{x^j}}_g$ is the coordinate expression for $g$ and $G'$ is defined similarly. For a matrix-valued one-form $A$, $\|A\|_{L^\infty(U,g)} := \esssup_{x \in U} \|A(x)\|_{2,g}$ and for a matrix-valued function, $\|A\|_{L^\infty(U)} := \esssup_{x \in U} \|A(x)\|_2$.
\end{lemma}

\correction{}{A proof of this lemma can be found in the appendix.}

\begin{remark}
While the fact that the difference between Christoffel symbols is bounded by $\|G - G'\|$ and $\|dG - dG'\|$ is fairly obvious, this theorem makes explicit how they relate to bounds on the $L^2$ operator norm of $\nabla - \nabla'$. Additionally, if $g$ is piecewise flat, there exists a coordinate chart $U \subseteq T \cap T'$ around any point $x \in T \cap T'$ such that $G = \lambda I$ for a constant $\lambda > 0$ and $\|G'^{-1}\|_{L^\infty(U)} = 1$, yielding a much simpler, and intrinsic, expression:
\begin{equation}\label{simplecomparecov}\|\nabla - \nabla'\|_{\mathcal{L}(L^2\Omega^{p,k}(U,g),L^2\Omega^{p+1,k}(U,g))} \le \frac{3(k+p)n^\frac{7}{2}}{2}\|dG'\|_{L^\infty(U,g)} = \frac{3(k+p)n^\frac{7}{2}}{2}\|\nabla g'\|_{L^\infty(U,g)}\end{equation} 

Here $g'$ is understood to be a piecewise smooth $(0,2)$-tensor and $\nabla g'$ a $(0,3)$-tensor, and $\|\nabla g'\|_{L^\infty(U,g)} := \esssup_{x \in U}\|\nabla g'|_x\|_g$.
\end{remark}

\subsection{Trace and Inverse Inequalities}

With all preparations done, we are finally ready to prove trace and inverse inequalities for simplices endowed with an arbitrary Riemannian metric. The usual scaling arguments are more complicated than usual, so complete proofs are provided. In this section, $T$ is a single $n$-simplex of $\mathcal{T}$, so $g$ is smooth on $T$.

\begin{theorem}[$H^1$ Trace Inequality for Riemannian Simplices]
\label{H1tracethrm}\;\\
Suppose $\alpha \in H^1\Omega^{p,k}(T,g)$, and assume that there exists an orientation-preserving diffeomorphism $f: \hat T \to T$ where $\hat T \subset \mathbb{R}^n$ is the standard simplex. Then $i_{\partial T}^*\alpha$ is well-defined and belongs to $L^2\Omega^{p,k}(\partial T,g)$, and there exists $\hat{C}$ depending only on $n$, $p$, and $k$ such that
\begin{align*}
\|&i_{\partial T}^*\alpha\|_{L^2(\partial T,g)} \\&\le \hat C  C_{g,\hat \delta}(T)^{k + p} C_{\hat \delta,g}(T)^{k+p + \frac{1}{2}}\sqrt{D_{g,\hat \delta}(T)D_{\hat \delta,g}(T)} \\
&\quad\cdot \bigg[\left(1 + C_{\hat \delta, g}(T)^{k+p + 1}C_{g,\hat \delta}(T)^{k+p + 1}\sqrt{D_{\hat \delta,g}(T)D_{g,\hat \delta}(T)} \|\hat \nabla g\|_{L^\infty(T,\hat \delta)}\right)\|\alpha\|_{L^2(T,g)} \\
&\quad\quad + C_{g,\hat \delta}(T)\|\nabla \alpha\|_{L^2(T,g)}\bigg],
\end{align*}
where $\hat \delta = f^{-*} \delta$ is the metric induced by $f$ and the standard Euclidean metric $\delta$ on $\hat T$, and $\hat\nabla$ is the associated connection.
\end{theorem}

\begin{remark}
In the case that $p = k = 0$, $\hat\nabla\alpha = d\alpha = \nabla\alpha$, and the term involving $\|\hat \nabla g\|$ can be removed as it is measuring the difference between $\nabla$ and $\hat \nabla$.
\end{remark}

\begin{proof}[Proof of Theorem \ref{H1tracethrm}]
We will take as given that there is a trace theorem for differential forms on $(\hat T, \delta)$. This is essentially because a form in this manifold is in $H^1$ if and only if each of its coefficients is in $H^1$ in the standard coordinates, and a trace theorem holds for each coefficient independently. Thus, setting $\hat \alpha = f^*\alpha$, $i_{\partial \hat T}^* \hat \alpha$ is well-defined and
\begin{equation}\label{basictraceeqn}
\|i_{\partial \hat T}^*\hat \alpha\|^2_{L^2(\partial \hat T,\delta)} \le \hat C \|\hat \alpha\|^2_{H^1(\hat T, \delta)}.
\end{equation}
Let $i_{\partial T}^*\alpha := f^{-*}i_{\partial \hat T}^*\hat \alpha$. This definition is consistent when $\alpha$ is continuous, because $f$ restricts to an embedding $e \hookrightarrow M$ for each facet $e \subset \partial T$. We will bound both sides of this inequality. Firstly, by Theorem \ref{L2_equiv},
\begin{align*}
\|i_{\partial \hat T}^*\hat\alpha\|^2_{L^2(\partial \hat T,\delta)} &= \|i_{\partial T}^*\alpha\|_{L^2(\partial T,\hat \delta)}^2 \\
&\ge \frac{1}{N_1^2C_{\hat \delta,g}(\partial T)^{2k + 2p}D_{g,\hat \delta}(\partial T)}\|i_{\partial T}^*\alpha\|_{L^2(\partial T,g)}^2,
\end{align*}
where $N_1$ is a constant depending only on $n$, $p$, and $k$.
Secondly,
\begin{align*}
\|\hat{\alpha}\|^2_{H^1(\hat T, \delta)} &= \|\alpha\|^2_{L^2(T,\hat \delta)} + \|\hat\nabla\alpha\|^2_{L^2(T,\hat \delta)}\\
&\le N_2^2C_{g,\hat \delta}(T)^{2k + 2p}D_{\hat \delta,g}(T)\left(\|\alpha\|_{L^2(T,g)}^2 + C_{g,\hat \delta}(T)^2\|\hat\nabla\alpha\|_{L^2(T,g)}^2\right),
\end{align*}
where $N_2$ is another constant depending only on $n$, $p$, and $k$.
Plugging both of these inequalities into inequality (\ref{basictraceeqn}), we obtain
\begin{equation}\label{H1traceintermediate1}
\begin{split}
\|i_{\partial T}^*\alpha\|^2_{L^2(\partial T, g)} &\le \hat C N_1^2N_2^2C_{\hat \delta,g}(\partial T)^{2(k+p)}C_{g,\hat \delta}(T)^{2(k+p)}D_{\hat \delta,g}(T)D_{g,\hat \delta}(\partial T) \\ &\quad\cdot \left(\|\alpha\|_{L^2(T,g)}^2 + C_{g,\hat \delta}(T)^2\|\hat\nabla\alpha\|_{L^2(T,g)}^2\right).
\end{split}
\end{equation}
It is clearly the case that $C_{g_1,g_2}(\partial T) \le C_{g_1,g_2}(T)$ for any two continuous metrics $g_1, g_2$ on $T$. We can also bound $D_{g,\hat \delta}(\partial T)$ in terms of $D_{g,\hat \delta}(T)$. In a relatively open neighborhood $U \subset T$ of a point $x_0$ of $\partial T$ which is contained in a smooth component of $\partial T$, $dV_{g}$ can be written as $\theta^1\wedge\star_g\theta^1$, where $\|\theta^1\|_g = 1$ and $\ker \theta^1|_x = T_x\partial T$ for all $x \in U \cap \partial T$. Evidently, $i_{\partial T}^*\star_g \theta^1$ is the $g$ volume form on $U \cap \partial T$.

At each $x$ in $U$, we have the following inequality:
\begin{align*}
\sup_{W_1,\dots,W_n \in T_xT\backslash\{0\}}&\frac{|dV_{g}(W_1,\dots,W_n)|}{\|W_1\|_{\hat \delta}\dots\|W_n\|_{\hat \delta}} \\&\ge \sup_{W_2,\dots,W_n \in \ker \theta^1|_x\backslash\{0\}}\sup_{W_1 \notin \ker \theta^1|_x}\frac{|\theta^1(W_1)| \cdot |\star_g\theta^1(W_2,\dots,W_n)|}{\|W_1\|_{\hat \delta}\|W_2\|_{\hat \delta} \dots\|W_n\|_{\hat \delta}}\\
&= \|\theta^1|_x\|_{\hat \delta}\|\star_g\theta^1|_x\|_{\hat \delta}.
\end{align*}
Therefore, we have
\[
D_{g,\hat \delta}(U) \ge D_{g,\hat \delta}(U \cap \partial T)\inf_{x \in U}\|\theta^1|_x\|_{\hat \delta}  \ge \frac{D_{g,\hat \delta}(U \cap \partial T)}{C_{\hat \delta,g}(U)}. 
\]
This local inequality implies a global inequality
\[
D_{g,\hat \delta}(\partial T) \le C_{\hat \delta,g}(T)D_{g,\hat \delta}(T).
\]
Plugging this into inequality (\ref{H1traceintermediate1}), we derive
\begin{align*}
\|i_{\partial T}^*\alpha\|_{L^2(\partial T,g)}^2 \le \hat C^2 N_1^2N_2^2 &C_{\hat \delta,g}(T)^{2(k+p) + 1}C_{g,\hat \delta}(T)^{2(k+p)} \\&\cdot D_{\hat \delta,g}(T)D_{g,\hat \delta}(T)\left(\|\alpha\|_{L^2(T,g)}^2 + C_{g,\hat \delta}(T)^2\|\hat\nabla\alpha\|_{L^2(T,g)}^2\right),
\end{align*}
or, more conveniently,
\begin{align}
\|i_{\partial T}^*\alpha\|_{L^2(\partial T,g)} \le \hat C N_1N_2 & C_{\hat \delta,g}(T)^{k+p + \frac{1}{2}}C_{g,\hat \delta}(T)^{k+p} \notag \\
&\cdot \sqrt{D_{\hat \delta,g}(T)D_{g,\hat \delta}(T)}\left(\|\alpha\|_{L^2(T,g)} + C_{g,\hat \delta}(T)\|\hat\nabla\alpha\|_{L^2(T,g)}\right).\label{tracebounds}
\end{align}
Lastly, we can use Lemma \ref{compare_cov} to assert that 
\[
\|\hat\nabla \alpha\|_{L^2(T,g)} \le \|\nabla\alpha\|_{L^2(T,g)} + \|\hat\nabla - \nabla\|_{\mathcal{L}(L^2\Omega^{p,k}(T,g),L^2\Omega^{p+1,k}(T,g))}\|\alpha\|_{L^2(T,g)}
.\]

Since $\hat \delta$ is flat on $T$, we have by inequality \ref{simplecomparecov} that $\|\hat \nabla - \nabla\|_{\mathcal{L}(L^2\Omega^{p,k}(T,\hat \delta),L^2\Omega^{p+1,k}(T,\hat \delta)} \le \frac{3(k+p)n^\frac{7}{2}}{2}\|\hat \nabla g\|_{L^\infty(T,\hat \delta)}$. To apply this inequality, we'll use
\begin{align*}
&\|\hat \nabla - \nabla\|_{\mathcal{L}(L^2\Omega^{p,k}(T,g),L^2\Omega^{p+1,k}(T,g))} \\ 
&\le \frac{\sup_{\beta \in L^2\Omega^{p+1,k}(T,g) \backslash \{0\}} \frac{\|\beta\|_{L^2(T,g)}}{\|\beta\|_{L^2(T,\hat \delta)}}}{\inf_{\beta \in L^2\Omega^{p,k}(T,g) \backslash \{0\}}\frac{\|\beta\|_{L^2(T,g)}}{\|\beta\|_{L^2(T,\hat \delta)}}} \|\hat \nabla - \nabla\|_{\mathcal{L}(L^2\Omega^{p,k}(T,\hat \delta),L^2\Omega^{p+1,k}(T,\hat \delta))} \\
&\le \hat C_1\hat C_2 C_{\hat \delta,g}(T)^{k+p + 1}C_{g,\hat \delta}^{k+p}(T)\sqrt{D_{\hat \delta,g}(T)D_{g,\hat \delta}(T)}\|\hat \nabla - \nabla\|_{\mathcal{L}(L^2\Omega^{p,k}(T,\hat \delta),L^2\Omega^{p+1,k}(T,\hat \delta))} \\
&\le \hat C_1 \hat C_2 \frac{3(k+p)n^\frac{7}{2}}{2}C_{\hat \delta,g}(T)^{k+p+1}C_{g,\hat \delta}(T)^{k+p}\sqrt{D_{\hat \delta,g}(T)D_{g,\hat \delta}(T)} \|\hat \nabla g\|_{L^\infty(T,\hat \delta)},
\end{align*}
where $\hat C_1$ and $\hat C_2$ depend only on $n$, $p$, and $k$.

Plugging this into inequality \ref{tracebounds} and grouping together constant multiples, we obtain the desired bound.

\end{proof}
Next is the well-known inverse inequality.

\begin{theorem}[Inverse Inequality for Riemannian Simplices]
Let $\hat T$, $T$, $\delta$, $\hat \delta$, $g$, $\nabla$, $\hat{\nabla}$, and $f$ be as in Theorem \ref{H1tracethrm}, and let $\hat V \subset H(d)\Omega^k(\hat T,\delta)$ be a finite-dimensional subspace and let $V_h = f^{-*}(\hat V) \subset H(d)\Omega^k(T,g)$. Then there exists $\hat C'$ depending only on $n$, $k$, and $\hat V$ such that for all $u_h \in V_h$,
\[
\|du_h\|_{L^2(T,g)} \le \hat C' C_{\hat \delta,g}(T)^{k + 1}C_{g,\hat \delta}(T)^{k}\sqrt{D_{\hat \delta,g}(T)D_{g,\hat \delta}(T)}\|u_h\|_{L^2(T,g)}.
\]
Likewise, if $\hat{V} \subset H^1\Omega^{p,k}(\hat T,\delta)$ is finite-dimensional and $V_h = f^{-*}(\hat{V}) \subset H^1\Omega^{p,k}(T,g)$, then there exists $\hat C$ depending only on $n$, $p$, $k$, and $\hat V$ such that for all $u_h \in V_h$,
\begin{align*}
\|\nabla u_h\|_{L^2(T,g)} \le \hat CC_{ \hat \delta, g}(T)^{k+p + 1}&C_{g,\hat \delta}(T)^{k+p}\sqrt{D_{\hat \delta,g}(T)D_{g,\hat \delta}(T)}[1 + \|\hat \nabla g\|_{L^2(T,\hat \delta)}]\|u_h\|_{L^2(T,g)}.
\end{align*}

\end{theorem}

\begin{proof}
The proofs of both claims are quite similar, so we will prove the second claim and then explain how the same proof idea carries over to the first claim. Let $\hat u_h = f^*u_h$. We start with the observation that any two norms on the finite-dimensional subspace $\hat V$ are equivalent, so there exists $\hat C$ independent of $\hat u_h$ such that 
\begin{equation}\label{basicinverseeqn}\|\nabla_\delta\hat u_h\|_{L^2(\hat T, \delta)} \le \|\hat u_h\|_{H^1(\hat T, \delta)} \le \hat C \|\hat{u}_h\|_{L^2(\hat T, \delta)}.\end{equation}
We'll bound both sides.
\begin{align*}
\|\nabla_\delta\hat u_h\|_{L^2(\hat T, \delta)} &= \|\hat\nabla u_h\|_{L^2(T,\hat \delta)} \\
&\ge \frac{1}{N_1C_{\hat \delta,g}(T)^{k+p + 1}\sqrt{D_{g,\hat \delta}(T)}}\|\hat\nabla u_h\|_{L^2(T,g)}\\
&\ge \frac{1}{N_1C_{\hat \delta,g}(T)^{k+p + 1}\sqrt{D_{g,\hat \delta}(T)}} \\
&\quad\quad\quad\quad\cdot \left[\|\nabla u_h\|_{L^2(T,g)} - \|\hat\nabla - \nabla\|_{\mathcal{L}(L^2\Omega^{p,k}(T,g),L^2\Omega^{p+1,k}(T,g))}\|u_h\|_{L^2(T,g)}\right],
\\
\|\hat u_h\|_{L^2(\hat T, \delta)} &= \|u_h\|_{L^2(T,\hat \delta)}\\
&\le N_2C_{g,\hat \delta}(T)^{k+p}\sqrt{D_{\hat \delta,g}(T)}\|u_h\|_{L^2(T,g)}.
\end{align*}
Plugging these inequalities into inequality (\ref{basicinverseeqn}), we get 
\begin{align*}
\|\nabla u_h\|_{L^2(T,g)} \le \Big[\hat C N_1N_2&C_{\hat \delta,g}(T)^{k+p + 1}C_{g,\hat \delta}(T)^{k+p}\sqrt{D_{\hat \delta,g}(T)D_{g,\hat \delta}(T)} \\
&+ \|\hat\nabla -\nabla\|_{\mathcal{L}(L^2\Omega^{p,k}(T,g),L^2\Omega^{p+1,k}(T,g))}\Big]\|u_h\|_{L^2(T,g)}.
\end{align*}

And, just as in the proof of Theorem \ref{H1tracethrm}, we can apply the inequality
\begin{align*}
\|\hat \nabla - \nabla\|_{\mathcal{L}(L^2\Omega^{p,k}(T,g),L^2\Omega^{p+1,k}(T,g))} \le \hat C_1 \hat C_2 \frac{3(k+p)n^\frac{7}{2}}{2}&C_{\hat \delta,g}(T)^{k+p + 1}C_{g,\hat \delta}(T)^{k+p}\\&\cdot\sqrt{D_{\hat \delta,g}(T)D_{g,\hat \delta}(T)} \|\hat \nabla g\|_{L^\infty(T,\hat \delta)},
\end{align*}
which gives the desired bound.

The proof of the $H(d)$ inverse inequality is nearly identical with $\nabla$ replaced with $d$, except that there is no difference between $\hat d$ and $d$.
\end{proof}

It is helpful to consider how these theorems relate to coordinate expressions. If $M = \mathbb{R}^n$ and $g$ is the Euclidean metric then, expressed in coordinates on the reference simplex, we have $\hat\nabla g = \hat\nabla (JJ^\intercal)$ where $J = df$ is the Jacobian, and $C_{g,\hat\delta}(T) = \sup_{x \in \hat{T}}\|J(x)\|_2$, $C_{\hat\delta,g}(T) = \sup_{x \in \hat{T}}\|J^{-1}(x)\|_2$, $D_{g,\hat\delta}(T) = \sup_{x \in \hat{T}} \det(J(x))$, and $D_{\hat\delta,g}(T) = \sup_{x \in \hat{T}}\det(J^{-1}(x))$. Scaling arguments for the case when $g$ is the Euclidean metric are already well established in the literature, even when $f$ is not affine~\cite{bernardi1989optimal,kawecki2019finiteelementtheorycurved,evans2011icesreport1117}. However, usually such arguments lack a term analogous to the $\hat\nabla g$ term because they consider real-valued functions rather than tensors.

\subsection{Riemannian Shape-Regularity}\label{shapereg_section}

The issue of mesh quality on a manifold is subtle, as not only the geometry of a simplex, but also the quality of its embedding map $f_{h,T}$ must be considered. For the purpose of computational ease, we have chosen to introduce  new definitions of shape-regularity and shape constant. These \correction{are equivalent to}{generalize} the usual definitions\correction{ if each $T$ is a flat simplex in some Euclidean space and the map $f_{h,T}$ is affine}{}, and they have some similarities in their role in the trace and inverse inequalities.

\begin{definition}\label{shaperegdef}
A family of manifolds with Regge metrics $\{(M_h,\mathcal{T}_h,g_h)\}_{h \in S}$ (where $S \subset \mathbb{R}$) has the \emph{Riemannian shape-regularity property} if there exists a number $0 < K < \infty$ independent of $h$ and orientation-preserving diffeomorphisms $f_{h,T}: \hat T \to T$ for each $T \in \mathcal{T}_h$ such that $C_{g_h,\hat \delta}(T)C_{\hat \delta,g_h}(T) \le K$ for all $h \in S, T \in \mathcal{T}_h$.
\end{definition}

Here, as before, $\hat{T}$ denotes the standard simplex and $\hat \delta = f_{h,T}^{-*}\delta$ is the pushforward of the standard metric on $\hat T$.

The quantity $C_{g_h,\hat \delta}(T)C_{\hat \delta,g_h}(T)$ simultaneously measures how much $g_h$ differs from a scaled version of the Euclidean metric $\hat \delta$, and how much $f_{h,T}$ differs from a pure rigid motion and a scaling. As the next lemma will show, \correction{if a family of meshes on a compact manifold is piecewise-affine in some smooth curvilinear coordinates and shape-regular (in the Euclidean sense) when viewed in those coordinates, and the metrics $g_h$ differ from a smooth metric $g$ by a bounded amount, then the family of meshes is also shape-regular in the Riemannian sense.}{Euclidean shape-regularity in coordinates is sufficient for Riemannian shape-regularity on compact manifolds. However, this implication is one-way: a family of meshes can fail to have any coordinates in which it is Euclidean shape-regular, and still be Riemannian shape-regular. For instance, there may be no coordinates so that, expressed in coordinates, $f_{h,T}$ is affine for all $h,T$.}

\begin{lemma}
\label{shapereg}
Suppose $(M,g)$ is a compact Riemannian manifold with corners, and let $\{U_i\}_{i = 1}^N$ be a finite cover of $M$ by open sets each possessing a smooth embedding $\phi_i: U_i \to \mathbb{R}^{K_i}$ for some integer $K_i$, and let $V_i \subset U_i$ be another cover where each $V_i$ is closed. If $\{(M,\mathcal{T}_h,g_h)\}_{h \in S}$ is a manifold supporting a family of Regge metrics such that

\begin{enumerate}
\item{there exists $C < \infty$ such that $C_{g,g_h}(M)C_{g_h,g}(M) \le C$ for all $h \in S$,}
\item{every $n$-simplex $T$ is contained completely in $V_i$ for some index $i$,} 
\item{for each $h \in S, T \in \mathcal{T}_h$ there exists a one-to-one affine map $\tilde{f}_{h,T}: \mathbb{R}^n \to \mathbb{R}^{K_i}$ such that $\tilde{f}_{h,T}(\hat T) = \phi_i(T)$ (by abuse of notation, we'll also use $\tilde{f}_{h,T}$ to refer to the linear map $v \mapsto \tilde{f}_{h,T}(v) - \tilde{f}_{h,T}(0)$),}

\item{there exists $0 < K < \infty$ such that $\frac{\sup_{v \ne 0}\frac{\|\tilde{f}_{h,T}v\|}{\|v\|}}{\inf_{w \ne 0}\frac{\|\tilde{f}_{h,T}w\|}{\|w\|}} \le K$ for all $h \in S, T \in \mathcal{T}_h$,}
\end{enumerate}
 then $\{(M,\mathcal{T}_h,g_h)\}_{h \in S}$ has the Riemannian shape-regularity property.
\end{lemma}

Condition 4 is equivalent to the usual shape-regularity property $\frac{h_T}{\rho_T} \le K$ in the local Euclidean metric, which involves the Euclidean diameter $h_T$ of $T$ and the Euclidean diameter $\rho_T$ of the largest ball contained in $T$.

\begin{proof}[Proof of Lemma \ref{shapereg}]
Each $U_i$ has an induced flat metric $\delta_{U_i} := \phi_i^*(\sum_j dy^j\otimes dy^j)$, and each simplex $T \subseteq V_i$ has the induced flat metric $\hat\delta := (\tilde{f}_{h,T}^{-1} \circ \phi_i)^*(\sum_j dx^j \otimes dx^j)$. For a given $w \in \mathbb{R}^n$ and $x \in T$, we can define a corresponding vector $W = d(\phi_i^{-1} \circ \tilde{f}_{h,T})(w^j\pdiff{}{x^j}) \in T_x(T)$. Then $\|w\| = \|W\|_{\hat \delta}$ and $\|\tilde{f}_{h,T}w\| = \|W\|_{\delta_{U_i}}$. Since any vector $W \in T_x(T)$ can be expressed this way, condition 4 is equivalent to $C_{\delta_{U_i},\hat \delta}(T)C_{\hat \delta,\delta_{U_i}}(T) \le K$. 

Clearly from the definition of $C_{\hat \delta,g_h}(T)$, 
\[
C_{\hat \delta,g_h}(T) \le C_{g,g_h}(T)C_{ \delta_{U_i},g}(T)C_{\hat \delta,\delta_{U_i}}(T)
\]
and
\[
C_{g_h,\hat \delta}(T) \le C_{g_h,g}(T)C_{g, \delta_{U_i}}(T)C_{ \delta_{U_i},\hat \delta}(T).
\]
Since $V_i$ is closed and $M$ is compact, $V_i$ is compact, so $g$ and $ \delta_{U_i}$ are quasi-isometric on $V_i$. Let $C' := \max_{i} (C_{g, \delta_{U_i}}(V_i)C_{ \delta_{U_i},g}(V_i))$. Then, we have
\[
C_{\hat \delta,g_h}(T)C_{g_h,\hat \delta}(T) \le C C' K,
\]
which is exactly what we need for the family to be Riemannian shape-regular, using $f_{h,T} := \phi_i^{-1} \circ \tilde{f}_{h,T}$ for the maps $\hat T \to T$.
\end{proof}

\begin{figure}
    \centering
    \includegraphics[width=0.5\linewidth]{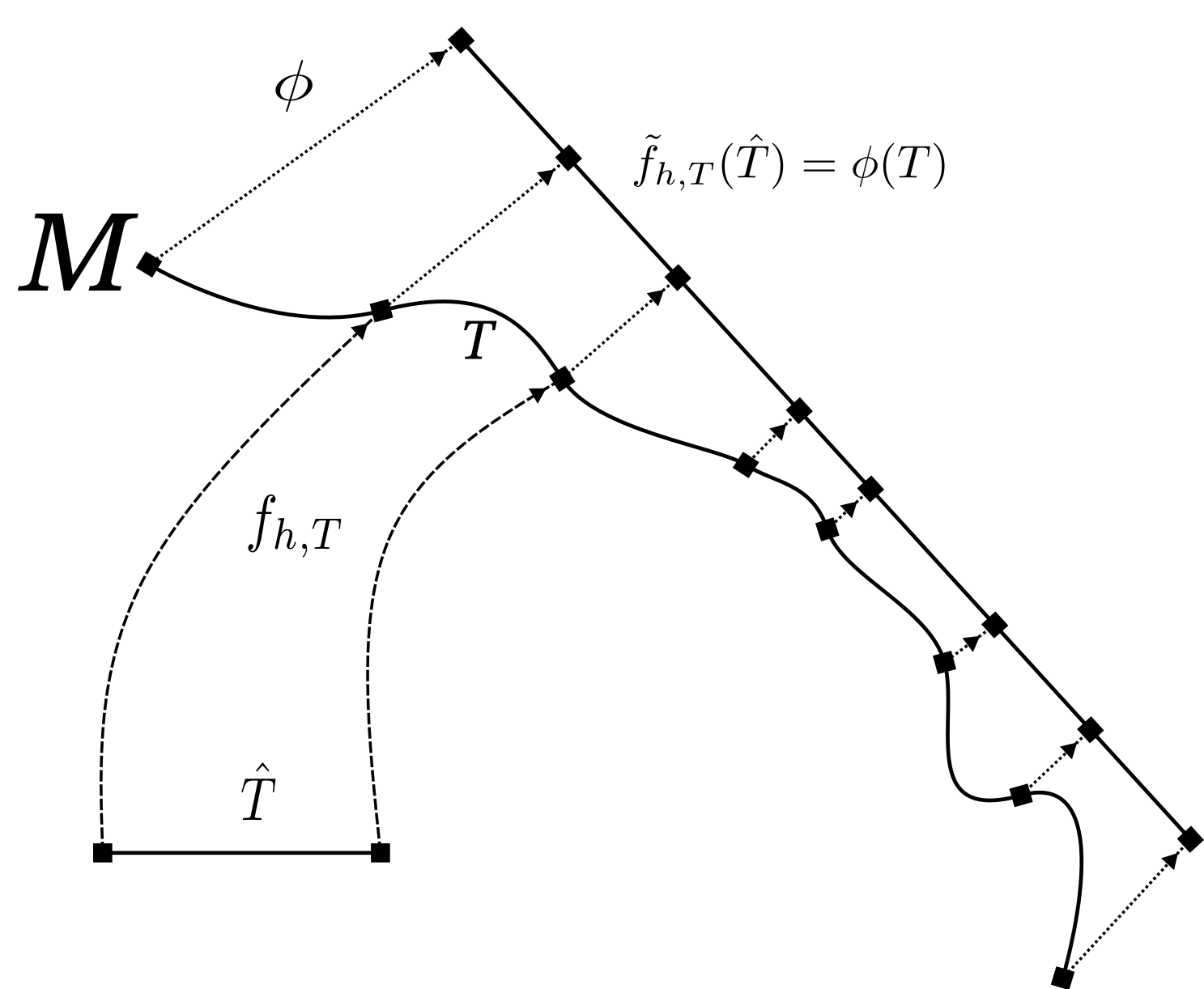}
    \caption{Example in which $M$ is a smooth curve in $\mathbb{R}^2$ and $\phi: M \to \mathbb{R}^2$ is a projection onto a straight line. The maps $f_{h,T}: [0,1] \to M$ can be taken as compositions of the affine maps $\tilde{f}_{h,T}: \hat{T} \to \phi(T)$ with the nonlinear map $\phi^{-1}$.}
    \label{fig:placeholder}
\end{figure}

We can also introduce a new definition of quasi-uniformity. In order for this definition to be meaningful in the same way it is in the Euclidean case, $h$ can no longer be an abstract parameter. Instead, we will reparameterize the family $\{(M_h,\mathcal{T}_h,g_h)\}_{h \in S}$ so that $\max_{T \subseteq M_h} C_{g_h,\hat{\delta}}(T) = h$. This is analogous to using $h$ to represent the maximum diameter of a simplex.

\begin{definition}\label{quasiuniformdef}
Let $\{(M_h,\mathcal{T}_h,g_h)\}_{h \in S}$ and the maps $f_{h,T}: \hat{T} \to M_h$ be as in Definition (\ref{shaperegdef}), such that $\max_{T \subseteq M_h} C_{g_h,\hat\delta}(T) = h$. If there exists a number $0 < K' < \infty$ such that $C_{\hat\delta,g_h}(T)\le K'h^{-1}$ for all $h \in S, T \in \mathcal{T}_h$ then we will say that the family of manifolds and Regge metrics has the \emph{Riemannian Quasi-Uniformity Property}.
\end{definition}

The quantity $C_{\hat\delta,g_h}(T)$ measures how much the map $f_{h,T}$ can shrink the length of a curve (measured in the $g_h$ metric) compared to its Euclidean length in the standard unit simplex. By bounding it from above, we assert that the simplices  do not ``shrink too fast'' relative to $h$. As before, the usual notion of quasi-uniformity in curvilinear coordinates implies Riemannian\correction{-}{} quasi-uniformity.

\begin{lemma}\label{quasiuniform}
With the assumptions of Lemma \ref{shapereg}, and the additional assumptions:

\begin{enumerate}
\setcounter{enumi}{4}
\item{$C_{g,g_h}(T) \le C$ for all $h \in S, T \in \mathcal{T}_h$},
\item{there exists $0 < K' < \infty$ such that $0 < \frac{1}{K'}h \le \inf_{v \ne 0} \frac{\|\tilde{f}_{h,T}v\|}{\|v\|}$ for all $h \in S, T \in \mathcal{T}_h$},
\end{enumerate}
then $\{(M,\mathcal{T}_h,g_h)\}_{h \in S}$ has the Riemannian quasi-uniformity property.
\end{lemma}
Condition 6 is equivalent to the usual quasi-uniformity property $\frac{1}{K'}h \le \rho_T$.
\begin{proof}[Proof of lemma \ref{quasiuniform}]
We already know that $C_{\hat \delta,g_h}(T) \le C_{\hat \delta,\delta_{U_i}}(T)C_{\delta_{U_i},g}(T)C_{g,g_h}(T)$ and $C_{\hat \delta,\delta_{U_i}}(T) = \sup_{v \ne 0} \frac{\|v\|}{\|\tilde{f}_{h,T}v\|} = \frac{1}{\inf_{v \ne 0} \frac{\|\tilde{f}_{h,T}v\|}{\|v\|}} \le \frac{1}{\frac{1}{K'}h}$. Therefore $C_{\hat \delta,g_h}(T) \le \max_iC_{\delta_{U_i},g}(V_i)C K'h^{-1}$.
\end{proof}

\begin{remark}\label{weakercoords}
    The maps $\phi_i$ do not need to be smooth embeddings on their whole domain; they merely need to be smooth embeddings when restricted to each simplex. An option for generalization would be to instead use a family of maps $\{\phi_{i,h}\}_{h \in S}$ such that $\phi_{i,h}: U_i \to \mathbb{R}^{K_i}$ is a piecewise-smooth embedding for $\mathcal{T}_h$, and there exist smooth embeddings $\phi_i: U_i \to \mathbb{R}^{K_i}$ and a constant $C$ such that $\|d\phi_{i,h}^j -  d\phi_i^j\|_{L^\infty(V_i,g)} \le C$ for all $h \in S$ and each coordinate $j = 1, \dots, K_i$. If $C$ is small enough, $C_{\delta_{U_i,h},g}(V_i)$ and $C_{g,\delta_{U_i,h}}(V_i)$ are both uniformly bounded, which would suffice to show shape-regularity and quasi-uniformity using the maps $f_{h,T} = \phi_{i,h}^{-1} \circ \hat{f}_{h,T}$. This could be useful, for instance, if $M$ is a smooth hypersurface in $\mathbb{R}^{n+1}$, $g$ is the metric on $M$ inherited from $\mathbb{R}^{n+1}$, and a piecewise-flat triangulated surface is used to approximate the smooth surface. The mesh $\mathcal{T}_h$ could be obtained by projecting from the approximate surface to the true surface, and the maps $\phi_{i,h}$ could be defined as the inverse of this projection map, with codomain $\mathbb{R}^{n+1}$.  The nearby smooth maps $\phi_i$ could each be defined as the inclusion map $M \hookrightarrow \mathbb{R}^{n+1}$. It must be stressed that in this situation, the approximate metrics $g_h$ need not necessarily be obtained from $\phi_{i,h}$, but setting $g_h$ equal to the pullback under $\phi_{i,h}$ of the Euclidean metric on $\mathbb{R}^{n+1}$ would yield a first-order approximation of $g$.
\end{remark}

To emphasize the roles of the different quantities, we will define $h_T := C_{g_h,\hat \delta}(T)$ and $\rho_{h,T} := \frac{1}{C_{\hat \delta, g_h}(T)}$. The Riemannian shape-regularity property can be rephrased as $\frac{h_T}{\rho_{h,T}} \le K$ and the Riemannian quasi-uniformity property can be rephrased as ${\rho_{h,T}}^{-1} \le K'h^{-1}$. We will also distinguish the bound on volume deformation by setting $K_V := \max_{T \subseteq M} D_{g_h,\hat \delta}(T)D_{\hat \delta,g_h}(T)$, since volumes typically transform differently than lengths in the scaling arguments. Note also that $\rho_{h,T} \le h_T$ since $1 = C_{g_h,g_h}(T) \le C_{g_h,\hat \delta}(T)C_{\hat\delta,g_h}(T) \le K$.

The value of Riemannian shape-regularity even in the circumstances when Lemmas (\ref{shapereg}) and (\ref{quasiuniform}) apply is that the resulting estimates for $K$, $K_V$, and $K'$ are overestimates unless the coordinate maps happen to be isometric embeddings for the smooth metric $g$. This discrepancy is especially problematic if $g$ is highly curved.

\begin{corollary}[Trace Inequality for Shape-regular Simplicial Complexes]
\label{shaperegtrace}
If $\{(M_h,\mathcal{T}_h,g_h)\}_{h \in S}$ is Riemannian shape-regular, then there exists $\hat C$ independent of $h,T, \alpha$ such that, for all $T \in \mathcal{T}_h$ and $\alpha \in H^1\Omega^{p,k}(T,g_h)$,
\begin{align*}
\|i_{\partial T}^*\alpha\|_{L^2(\partial T,g_h)} \le \hat C \sqrt{K_V}K^{k+p}\Big[&{\rho_{h,T}}^{-\frac{1}{2}}\Big(1 + K^{k+p+1}\sqrt{K_V} \|\hat \nabla g_h\|_{L^\infty(T,\hat \delta)}\Big)\|\alpha\|_{L^2(T,g_h)} \\
&+ K^\frac{1}{2}h_T^\frac{1}{2}\|\nabla_h \alpha\|_{L^2(T,g_h)}\Big].
\end{align*}
\end{corollary}

\begin{corollary}[Inverse Inequality for Shape-regular Simplicial Complexes] \label{shapereginverse}
If $\{(M_h,\mathcal{T}_h,g_h)\}_{h \in S}$ is Riemannian shape-regular, $\hat V \subset H(d)\Omega^k(\hat T, \delta)$ is finite-dimensional, $V_h := f_{h,T}^{-*}(\hat V)$, then there exists $\hat C$ depending only on $\hat{V}$, $n$, and $k$ such that for all $u_h \in V_h$,
\[
\|du_h\|_{L^2(T,g_h)} \le \hat C \sqrt{K_V}K^{k+1}{\rho_{h,T}}^{-1}\|u_h\|_{L^2(T,g_h)}.
\]

Likewise if $\hat V \subset H^1\Omega^{p,k}(\hat T,\delta)$ is finite-dimensional, $V_h := f_{h,T}^{-*}(\hat V)$,  then there exists $\hat C$ depending only on $\hat{V}$, $n$, $k$, and $p$ such that for all $u_h \in V_h$,
\[
\|\nabla_h u_h\|_{L^2(T,g_h)} \le \hat C\sqrt{K_V}K^{k+p}{\rho_{h,T}}^{-1}\left(1 + \|\hat \nabla g_h\|_{L^\infty(T,\hat \delta)}\right)\|u_h\|_{L^2(T,g_h)}.
\]
\end{corollary}

\section{Error Analysis of the Hodge-Laplace Problem on Manifolds with Regge Metrics}
In this section we will work out finite element a priori error analysis for the Hodge-Laplace problem on manifolds with Regge metrics, building on the foundation of geometric variational crimes \cite{Holst-Stern}. This section is mostly a straightforward application of existing theory. \correction{}{The main new component is Theorem \ref{Jhbounds}.}

Briefly, we start with a compact smooth Riemannian manifold with corners $(M,g)$ and the $L^2$-de Rham complex of differential forms, on which the Hodge-Laplace problem can be posed in the strong form (for a source term $f \in L^2\Omega^k(M,g)$):
\begin{equation}\label{poisson_strong}
(d\delta + \delta d)u = f \quad \text{on}\; \mathring{M}, \end{equation}
\begin{equation}\label{poisson_strong_boundary}
i_{\partial M}^*\star u = i_{\partial M}^*\star du = 0.
\end{equation}

The space of harmonic forms $\mathfrak{H}^k(M,g) := \{\alpha \in H(d)\Omega^k(M,g) \cap H(\delta)\Omega^k_\mathbf{n}(M,g) : d\alpha = \delta \alpha = 0\}$ can be shown to be equal to the closure of the set $\{u \in C^\infty\Omega^k(M) : (d\delta + \delta d)u = 0, i_{\partial M}^*\star u = i_{\partial M}^* \star du = 0\}$ in $H^1\Omega^k(M,g)$ \cite{Schwarz}. This is significant as the solution is only unique up to addition by a harmonic form, and the fact that the space of harmonic forms is compactly embedded in $L^2$ is a key component in proving the well-posedness of weak formulations of the problem.

There is a mixed weak formulation of this problem, which is well-posed \cite{Arnold,Arnold-Winther}: Find $(\sigma,u,p) \in H(d)\Omega^{k-1}(M,g) \times H(d)\Omega^k(M,g) \times \mathfrak{H}^k(M,g)$ such that, for all $(\tau,v,q) \in H(d)\Omega^{k-1}(M,g) \times H(d)\Omega^k(M,g) \times \mathfrak{H}^k(M,g)$,
\begin{align}\label{contproblem}
\begin{split}\iprod{\sigma,\tau}_{L^2(M,g)} - \iprod{d\tau,u}_{L^2(M,g)} &= 0\\
\iprod{d\sigma,\tau}_{L^2(M,g)} + \iprod{du,dv}_{L^2(M,g)} + \iprod{p,v}_{L^2(M,g)} &= \iprod{f,v}_{L^2(M,g)}\\
\iprod{u,q}_{L^2(M,g)} &= 0.
\end{split}\end{align}

\begin{remark}\label{dualnorm}Note that $f$ does not necessarily need to be in $L^2\Omega^k(M,g)$, as the mixed weak formulation is well-posed for any element of $H(d)\Omega^k(M,g)^*$. From now on we will use $F$ to refer to such a functional. While this change is not necessarily meaningful for the strong formulation of the problem, it can still be useful in geometric applications, as we will see in the next section. \end{remark}

To approximate this boundary value problem, we would like to use a finite-dimensional subcomplex $V_h^k \subset H(d)\Omega^k(M,g_h)$ where $g_h$ is a Regge metric on a triangulation $\mathcal{T}_h$ which approximates $g$, such that $d(V_h^k) \subset V_h^{k+1}$. We will use $V_h^k$ to refer to this space with the $H(d)\Omega^k(M,g_h)$ norm, and $W_h^k$ to refer to the same space with the $L^2\Omega^k(M,g_h)$ norm. The space of discrete harmonic forms is given by $\mathfrak{H}^k_h := \{u \in V_h^k : du = 0, \iprod{u,d\eta}_{L^2(M,g_h)} = 0 \forall \eta \in V_h^{k-1} \}$. The discrete weak problem is then: given $F_h \in {V_h^k}^*$, find $(\sigma_h,u_h,p_h) \in V_h^{k-1} \times V_h^k \times \mathfrak{H}^k_h$ such that, for all $(\tau_h,v_h,q_h) \in V^{k-1}_h \times V^k_h \times \mathfrak{H}^k_h$,
\begin{align}\label{discproblem}
\begin{split}\iprod{\sigma_h,\tau_h}_{L^2(M,g_h)} - \iprod{d\tau_h,u_h}_{L^2(M,g_h)} &= 0,\\
\iprod{d\sigma_h,\tau_h}_{L^2(M,g_h)} + \iprod{du_h,dv_h}_{L^2(M,g_h)} + \iprod{p_h,v_h}_{L^2(M,g_h)} &= F_h(v_h),\\
\iprod{u_h,q_h}_{L^2(M,g_h)} &= 0.
\end{split}\end{align}

To analyze the stability and accuracy of this method in the framework of geometric variational crimes, we need to produce bounded cochain maps $i_h^k: W_h^k \to L^2\Omega^k(M,g)$ and $\pi_h^k: H(d)\Omega^k(M,g) \to V_h^k$ such that $i_h^k(W_h^k) \subset H(d)\Omega^k(M,g)$ and $\pi_h^k \circ i_h^k  = \mathrm{Id}$. Since the $L^2$ and $H(d)$ spaces of forms on a compact manifold with corners do not depend (as topological spaces) on the metric, $i_h^k$ can simply be the inclusion map $W_h^k \subset L^2\Omega^k(M,g)$, and $\pi_h^k$ can be a family of bounded cochain projection operators $H(d)\Omega^k(M,g) \to i_h^k (V_h^k)$. This will induce bounded cochain maps $W_h^k \to L^2\Omega^k(M,g)$ and $H(d)\Omega^k(M,g) \to V_h^k$ with the desired properties so long as $C_{g_h,g}(M)$ and $C_{g,g_h}(M)$ are uniformly bounded with respect to $h$. 

Specifically, 
\[\|i_h^k\|_{\mathcal{L}(W_h^k,L^2\Omega^k(M,g))} = \sup_{v_h \in W_h\backslash \{0\}} \frac{\|v_h\|_{L^2\Omega^k(M,g)}}{\|v_h\|_{L^2\Omega^k(M,g_h)}} \le \sqrt{n \choose k}C_{g_h,g}(M)^k\sqrt{D_{g,g_h}(M)}\]
and 
\begin{align*}\|\pi_h^k\|_{\mathcal{L}(H(d)\Omega^k(M,g),V_h^k)} &= \sup_{v \in H(d)\Omega^k(M,g)\backslash \ker \pi_h^k} \frac{\|\pi_h^k(v)\|_{H(d)(M,g_h)}}{\|\pi_h^k(v)\|_{H(d)(M,g)}}\frac{\|\pi_h^k(v)\|_{H(d)(M,g)}}{\|v\|_{H(d)(M,g)}}\\
	&\le \sup_{v_h \in V_h^k\backslash \{0\}}\frac{\|v_h\|_{H(d)(M,g_h)}}{\|v_h\|_{H(d)(M,g)}}\sup_{v \in H(d)\Omega^k(M,g)\backslash\{0\}} \frac{\|\pi_h^k(v)\|_{H(d)(M,g)}}{\|v\|_{H(d)(M,g)}}\\
	\le \sqrt{n \choose k}\max(&C_{g,g_h}(M)^k,C_{g,g_h}(M)^{k+1})\sqrt{D_{g_h,g}(M)}\|\pi^k_h\|_{\mathcal{L}(H(d)\Omega^k(M,g),H(d)\Omega^k(M,g))}.\end{align*}

We will take the existence of a family of commuting projection operators $\pi_h^k: H(d)\Omega^k(M,g) \to i_h^k(V_h^k)$ as given. Such operators have been constructed, mostly for Euclidean domains \cite{Christiansen-Winther} and somewhat recently for Riemannian manifolds \cite{Licht}. Their construction is not trivial. However, it is worth noting that the canonical degrees of freedom for the $\mathcal{P}_r\Lambda^k$ and $\mathcal{P}_r^-\Lambda^k$ spaces from finite element exterior calculus~\cite{Arnold-Winther} are completely independent of the metric when applied to smooth forms, as they only depend on integrals of smooth differential $j$-forms over $j$-simplices (for $j \ge k$) and a local homotopy operator $\kappa$. The difficulty is in constructing a smoothing operator $H(d)\Omega^k(M,g) \to C^\infty\Omega^k(M)$ that is bounded in the $H(d)$ norm, commutes with the exterior derivative, and preserves boundary conditions.

Theorem 3.9 in \cite{Holst-Stern} establishes that the discrete problem is well-posed with a uniform inf-sup constant, while Corollary 3.11 in \cite{Holst-Stern} establishes that the following error bound holds: 
\begin{align}\label{errbounds}
\begin{split}
&\|\sigma - \sigma_h\|_{H(d)(M,g)} + \|u - u_h\|_{H(d)(M,g)} + \|p - p_h\|_{L^2(M,g)}\\
&\le C\Bigg[\inf_{\tau \in V_h^{k - 1}} \| \sigma - \tau\|_{H(d)(M,g)} + \inf_{v \in V_h^k} \|u - v\|_{H(d)(M,g)} + \inf_{q \in V_h^k} \|p - q\|_{H(d)(M,g)}\\
&\quad\quad+\inf_{v \in V_h^k}\sup_{r \in \mathfrak H^k(M,g)\backslash \{0\}} \frac{\|(Id - \pi_h^k)r\|_{L^2(M,g)}}{\|r\|_{L^2(M,g)}}\|P_{\mathfrak{B}}u - v\|_{H(d)(M,g)}\\
&\quad\quad+\|F_h - F\|_{{V_h^k}^*} + \left(\|Id - J_h\|_{\mathcal{L}(W_h^{k-1},W_h^{k-1})} + \|Id - J_h\|_{\mathcal{L}(W_h^k,W_h^k)}\right)\|F\|_{H(d)\Omega^k(M,g)^*}\Bigg],
\end{split}\end{align}
where $J_h := {i_h^k}^* \circ i_h^k$, ${i_h^k}^*: L^2\Omega^k(M,g) \to W_h^k$ is the $L^2$ adjoint of $i_h^k$, and $P_\mathfrak{B}: L^2\Omega^k(M,g) \to L^2\Omega^k(M,g)$ is the orthogonal projection onto the closed subspace $\mathfrak{B} = \{ dv : v \in H(d)\Omega^{k-1}(M,g)\}$. The proof of Theorem 3.10 in \cite{Holst-Stern} can be carried out with minimal modifications using $F \in H(d)\Omega^k(M,g)^*$, $F_h \in H(d)\Omega^k(M,g_h)^*$, and dual norms thereof in place of $\iprod{i_h^*f,\cdot}_h$, $\iprod{f_h,\cdot}_h$, and $L^2$-norms thereof, respectively; this is because, as noted in Remark \ref{dualnorm}, the weak forms are all well-posed for this more general class of functionals, and the dual norm can be used in place of the Cauchy-Schwarz inequality when obtaining the bound $\iprod{f_h - i_h^* f, v_h}_h \le \|f_h - i_h^*f\|_h\|v_h\|_{V_h}$. Again, this is not necessarily meaningful for the original strong formulation (\ref{poisson_strong}), but it can still be useful in practice. 

Our main contribution that is not covered by existing literature will be putting bounds on $\|Id - J_h\|$ that depend only on intrinsic quantities associated to a Regge metric.
\begin{theorem}
\label{Jhbounds}
There exists $C$ depending only on $n$ and $k$ such that, if $\|g - g_h\|_{L^\infty(M,g_h)} < \frac{1}{C}$, then
\begin{equation}\|Id - J_h\|_{\mathcal{L}(W_h^k,W_h^k)} \le {n \choose k} C\left[2 + C\|g_h - g\|_{L^\infty(M,g_h)}\right]\|g_h - g\|_{L^\infty(M,g_h)},\end{equation}
where $\|g - g_h\|_{L^\infty(M,g_h)} = \esssup_{x \in M}\sup_{V,W \in T_x(M) \backslash \{0\}} \frac{|(g - g_h)(V,W)|}{\|V\|_{g_h}\|W\|_{g_h}}$.

\end{theorem}
\begin{remark}
Note that the $h$-dependent norms on the right-hand side of the above inequality can be replaced by $h$-independent norms using the fact that $\|g - g_h\|_{L^\infty(M,g_h)} \le C_{g_h,g}(M)^2\|g - g_h\|_{L^\infty(M,g)}$.
\end{remark}
\begin{proof}[Proof of Theorem \ref{Jhbounds}]

$J_h$ has an explicit formula in terms of the Hodge star operators for $g$ and $g_h$. First, we note that, from the definition of the $L^2(M,g)$ and $L^2(M,g_h)$ inner products, we have that
\[\sum_{T \subseteq M} \int_T u \wedge \star_h (\star_h^{-1} \star v) = \int_M u \wedge \star v = \iprod{u,v}_{L^2(M,g)} = \iprod{u,J_hv}_{L^2(M,g_h)} = \sum_{T \subseteq M} \int_T u \wedge \star_h J_h v.\]

So, the two maps $v \mapsto \iprod{u,\star_h^{-1}\star v}_{L^2(M,g_h)}$ and $v \mapsto \iprod{u,J_hv}_{L^2(M,g_h)}$ are identical for any $u \in W_h^k$. This is only possible if $J_h = \star_h^{-1}\star$.

Firstly we will define some quantities related to any pair of orthonormal coframes. Let $\{\theta^i\}_{i = 1}^n$ be a $g$-orthonormal coframe  and $\{\theta^i_h\}_{i = 1}^n$ be a $g_h$-orthonormal coframe, which induce the same orientation on $T_x^*(T)$. This gives us bases $\{\theta^I\}_{|I| = k}$ and $\{\theta^J_h\}_{|J| = k}$ for $\Lambda^k_x(M)$. They are related by the change of basis matrix (with each multi-index assigned an integer index) $\Theta^I_J := \iprod{\theta^I,\theta^J_h}_h$. For each $I$ there is also a unique multi-index called $[n]\backslash I$ and a quantity $\sigma(I)$ such that $\star \theta^I = (-1)^{\sigma(I)}\theta^{[n]\backslash I}$ and $\star_h \theta^I_h = (-1)^{\sigma(I)} \theta^{[n]\backslash I}_h$. We will use $\tilde\Theta$ to refer to the change of basis matrix for $\Lambda^{n-k}_x(M)$, i.e. $\tilde\Theta^K_L := \iprod{\theta^K,\theta_h^L}_h$ for all multi-indices $K$ and $L$, $|K| = |L| = n-k$. The enumeration of multi-indices can be selected so that the integer index corresponding to the multi-index $[n]\backslash I$ in $\tilde \Theta$ is the same as that for $I$ in $\Theta$. Since the dimension of $\Lambda^{n-k}_x(M)$ is the same as that of $\Lambda^k_x(M)$, the matrices $\Theta$ and $\tilde \Theta$ have the same shape.

Thus we can write:
\begin{align*}J_h|_x \theta^I_h &= \star_h^{-1}\star\sum_K(\Theta^{-1})^I_K\theta^K = \star_h^{-1}\sum_K(-1)^{\sigma(K)}(\Theta^{-1})^I_K\theta^{[n]\backslash K} \\
&= \sum_{K,J}(-1)^{\sigma(K) + \sigma(J)}(\Theta^{-1})^I_K\tilde\Theta^{[n]\backslash K}_{[n]\backslash J} \theta^J_h.\end{align*}

So, expressing the endomorphism $Id - J_h|_x$ as a matrix in the basis $\{\theta^I_h\}$, we get
\begin{equation}\label{Jheqn}(Id - J_h|_x)^I_J = \delta^I_J - (-1)^{\sigma(J)}\sum_K(-1)^{\sigma(K)}(\Theta^{-1})^I_K\tilde\Theta^{[n]\backslash K}_{[n]\backslash J}.\end{equation}

Now we pick a specific $\{\theta_h\}$ depending on $\{\theta\}$. Let the matrix $G$ be defined by $(G^{-1})^i_j = \iprod{\theta^i,\theta^j}_h$. Then let $E$ be the Cholesky factor of $G$, so $\theta^i_h = \sum_j E^i_j\theta^j$ gives a $g_h$-orthonormal coframing. Then $\Theta^I_J = \iprod{(E^{-1})^{i_1}_l\theta_h^l\wedge\dots\wedge (E^{-1})^{i_k}_m\theta_h^m,\theta^J_h}_h = \det([(E^{-1})^{i_a}_{j_b}])$. Likewise, $\tilde\Theta^K_L = \det([E^{k_a}_{l_b}])$.

Since the Cholesky factorization, inverse, and determinant are all smooth functions, by Taylor series expansion there exists a constant $C$ such that if $\|G - \delta\|_m < \frac{1}{C}$, then $\|\Theta^{-1} - \delta\|_m < C\|G - \delta\|_m$ and $\|\tilde\Theta - \delta\|_m \le C\|G - \delta\|_m$.
 
This constant $C$ depends only on the derivatives of the determinant and Cholesky square root maps at the matrices $[\delta^{i_a}_{j_b}]$ and $[\delta^{([n]\backslash I)_a}_{([n]\backslash J)_b}]$ for each multi-index $I$ and $J$, and the derivatives of the inverse map at the ${n \choose k}\times{n \choose k}$ and $k\times k$ identity matrices. Therefore it depends only on $n$ and $k$.
 
 Lastly, let $\tilde{\delta}^I_J := (-1)^{\sigma(J)}\delta^I_J$. Putting all this together with (\ref{Jheqn}) we can conclude 
\begin{align*}|(Id - J_h|_x)^I_J| &= \Big|\delta^I_J - (-1)^{\sigma(J)}\sum_K(\Theta^{-1}\tilde{\delta})^I_K\tilde\Theta^{[n]\backslash K}_{[n]\backslash J}\Big|\\
&= \big|\delta^I_J - (-1)^{\sigma(J)}[\tilde{\delta} + \tilde{\delta}(\tilde\Theta - \delta) + (\Theta^{-1} - \delta)\tilde{\delta} + (\Theta^{-1} - \delta)\tilde{\delta}(\tilde\Theta  - \delta)]^I_J\big|\\
&= \big|(-1)^{\sigma(J)}[\tilde{\delta}(\tilde\Theta - \delta) + (\Theta^{-1}-\delta)\tilde{\delta} + (\Theta^{-1} - \delta)\tilde{\delta}(\tilde\Theta - \delta)]^I_J\big|\\
&\le \esssup_{x \in M}C\left[2 + C\|G - \delta\|_m\right]\|G - \delta\|_m\\
&\le C\left[2 + C\|g - g_h\|_{L^\infty(M,g_h)}\right]\|g - g_h\|_{L^\infty(M,g_h)}.\end{align*}

Since $Id - J_h$ extends to a bundle endomorphism $\Lambda^k(T) \to \Lambda^k(T)$ on each simplex $T \subseteq M$, it is clearly the case that
\[\|Id - J_h\|_{\mathcal{L}(W_h^k,W_h^k)} \le \esssup_{x \in M} \|Id - J_h|_x\|_2 \le {n \choose k}\esssup_{x \in M} \max_{I,J}|(Id - J_h|_x)^I_J|.\]

This gives the desired bound on $\|Id - J_h\|_{\mathcal{L}(W_h^k,W_h^k)}$.
\end{proof}

Showing that a complex of finite element spaces $\{V_h^k\}_{k=0}^n$ and/or the approximated metric $g_h$ actually have the approximation properties that lead to a convergent discretization would be a complicated matter, but we will sidestep it by assuming we are in a situation of Lemma \ref{shapereg}: there exists a finite closed cover $\{W_i\}$ of $M$ such that each \correction{triangle}{simplex} $T \subseteq M$ lies entirely within $W_i$ for some $i$, and each $W_i$ itself lies within a coordinate chart $U_i$, and the maps $f_{h,T}$ defining the \correction{triangles}{simplices} $T \subseteq M$ are affine when expressed in these coordinates. In practice, this is a convenient way to describe a computational mesh on a manifold (see remark \ref{weakercoords} for an indication of how less-smooth coordinate maps could be used as well). For each $T \subseteq M$, we pick a specific $i(T)$ so that $T \subset W_{i(T)}$, and set $\|\cdot\|_E^2 := \sum_{T \subseteq M} \|\cdot\|_{L^2(T,\delta_{U_{i(T)}})}^2$. 

Approximation properties of finite element spaces are usually proved for flat simplices in Euclidean space, which would correspond to using the norm $\|\cdot\|_E$. By Theorem \ref{L2_equiv}, \\$\|\cdot\|_{L^2(M,g)} \le \sqrt{n \choose k}\max_iC_{\delta_{U_i},g}(W_i)^k\sqrt{D_{g,\delta_{U_i}}(W_i)}\|\cdot\|_E$. This constant does not depend on $g_h$, $\mathcal{T}_h$, or $V_h$, so approximation properties of $V_h$ proved in the context of local coordinates are sufficient.

Likewise, $\|g - g_h\|_{L^\infty(M,g)} \le \max_iC_{\delta_{U_i},g}(W_i)^2\|g - g_h\|_{L^\infty(W_i,\delta_{U_i})}$, so approximation properties of the metric $g_h$ may be proved using local coordinates as well.

\section{Calculation of Connection Forms as a Hodge-Laplace Problem}

As an application of the geometric methods developed in the previous sections of this paper to solve a geometric problem, we will approximate the connection 1-form $\alpha$ associated to a special frame field $(e_1,e_2)$ on a simply connected 2-dimensional manifold \correction{}{with corners equipped} with a smooth Riemannian metric, called $(M,g)$. Specifically, we wish to approximate $\alpha = A^1_2 := \iprod{\nabla e_2, e_1}$. We want the frame field to have the property that $\delta \alpha = 0$. In fact, we require an even stronger property, that $\alpha = \delta \omega$ for some 2-form $\omega$.

First, we note that there actually does exist such a frame. Let $(e_1,e_2)$ be \emph{any} smooth orthonormal frame on $(M,g)$ with corresponding connection form $B^1_2 = \beta = df + \delta \omega$, where the splitting is unique by the Hodge decomposition, and let $\Phi: M \to SO(2)$ be a map which rotates the frame $(e_1,e_2)$ smoothly on $M$ by the angle $\theta: M \to \mathbb{R}$. Then the frame $\Phi \cdot (e_1,e_2)$ has the connection form 
\[{A_\Phi}^1_2 = (\Phi^{-1}d\Phi)^1_2 + \mathrm{Ad}(\Phi)(A)^1_2,\] 
where $\mathrm{Ad}$ denotes the adjoint representation of $SO(2)$  on $\mathfrak{so}(2)$, defined by $\mathrm{Ad}(H)(w) := H w H^{-1}$.

Since $M$ is 2-dimensional, $\mathrm{Ad}$ is the trivial action. Additionally, the $\mathfrak{so}(2)$-valued 1-form $\Phi^{-1}d\Phi$ is, in coordinates, 
\[\begin{bmatrix}0 & -d\theta \\ d\theta & 0\end{bmatrix}.\]

Thus, if $\Phi$ is a rotation by $\theta = -f \mod 2\pi$, then ${A_\Phi}^1_2 = \delta \omega$, and $\delta {A_\Phi}^1_2 = \delta \delta \omega = 0$. The resulting frame $(\tilde{e_1},\tilde{e_2})$ is also unique up to a constant rotation.

The condition we have imposed on $A^1_2$ is reminiscent of the Coulomb gauge from electrodynamics, where $A$ would be the magnetic vector potential. While the physical significance of $A^1_2$ and the frame $(e_1,e_2)$ are not obvious, the idea is basically the same as gauge fixing.

To set up a weak form of the equation for $\alpha$, or more specifically $*\alpha$, we let $\omega = *u$, so that $\alpha = *du$. We then apply the fact that $d\alpha = KdA$, where $dA$ is the volume form and $K$ is the Gauss curvature. Thus, by applying integration by parts, we arrive at the weak form of our problem: find $u \in H(d)\Omega^0(M,g)$ such that, for all $v \in H(d)\Omega^0(M,g)$ and $c \in \ker d = \mathfrak{H}^0(M,g)$,
\begin{equation}\label{weakform}\iprod{du,dv}_{L_2(M,g)}  = \int_{\partial M} v\alpha - \int_M vKdA,\end{equation}
\begin{equation}\label{weakformbc}\iprod{u,c}_{L^2(M,g)} = 0. \end{equation}

This is in fact a specialization of the abstract Hodge-Laplace problem for the $L^2(M,g)$-de Rham complex. Note that the right-hand side of (\ref{weakform}) is zero when evaluated against any constant function $v \in \ker d$,  meaning this equation is well-posed as written.

We will describe a discretization of the above problem. Let $\mathcal{T}_h$ be a Riemannian shape-regular and quasi-uniform family of triangulations of $M$ (in the sense of Section 2.3) such that $h_T \le h$ for all $T \in \mathcal{T}_h$, and let $g_h$ be a Regge metric which is smooth on the interior of each simplex of $\mathcal{T}_h$ and approximates $g$. Using a finite-dimensional subspace $V_h \subset H(d)\Omega^0(M,g_h)$ such that for each $T \subseteq M$, $f_{h,T}^*(v_h) = \hat v_h$ for some $\hat v_h \in \hat V \subset C^2\Omega^0(\hat T)$ for a fixed subspace $\hat V$, the discrete problem is: find $u_h \in V_h$ such that, for all $v_h \in V_h$ and $c \in V_h \cap \ker d$,
\begin{equation}
\label{discreteform}\iprod{du_h,dv_h}_{L^2(M,g_h)} = \iprod{\iprod{{\alpha|_{\partial M}}_\mathrm{dist}(g_h),v_h}} -\iprod{\iprod{KdA_{\mathrm{dist}}(g_h),v_h}} 
\end{equation}
and
\begin{equation}\label{discreteformbc}\iprod{u_h,c}_{L^2(M,g_h)} = 0.\end{equation}
The functional $KdA_{\mathrm{dist}}$ above is the distributional Gaussian curvature, defined in \cite{BKGa22,GaNe22} as
\[
\iprod{\iprod{KdA_\mathrm{dist}(g_h),v_h}} := \sum_{T \subseteq M} \int_T v_hK_h dA_h + \sum_{\mathring{e} \subset \mathring{M}} \int_e v_h[\![k_hds_h]\!] + \sum_{p \in \mathring{M}} v_h(p)\Bigg(2\pi - \sum_{T \ni p}\theta_T(p)\Bigg).
\] 

Above, $k_h$ is the geodesic curvature of the edge $e$ as measured by $g_h$ and $\theta_T(p)$ is the interior angle of the triangle $T$ at the point $p \in \partial T$ as measured by $g_h$. The functional ${\alpha|_{\partial M}}_{\mathrm{dist}}$ is defined as
\[
\iprod{\iprod{{\alpha|_{\partial M}}_\mathrm{dist}(g_h),v_h}} := \sum_{e \subset \partial M} \int_e v_h(d\mu - k_hds_h) - \sum_{p \in \partial M} v_h(p)\Bigg(\pi - [\![\mu]\!]|_p - \sum_{T \ni p} \theta_T(p)\Bigg),
\]
where $\mu$ is the angle that the smooth frame $e_1$ makes with the outward normal vector of $\partial M$ as measured by the smooth metric $g$.

 The above definition of ${\alpha|_{\partial M}}_{\mathrm{dist}}$ is perhaps not very obvious. There are two ways to make sense of it; one involves proving a generalization of the Gauss-Bonnet theorem to manifolds with Regge metrics, as we have done in a preprint \cite{gawlik2025curvaturereggemetrics}, and the other amounts to observing that the method converges with this definition. The computed form $\alpha_h = \star_h du_h$ may be interpreted as a form which approximates, in each triangle, the connection form associated to a special $g_h$-orthogonal frame field $(e_1,e_2)_h$ which makes the same angle (measured with $g_h$) with the outward normal vector as $(e_1,e_2)$ does (measured with $g$) and such that its associated connection one-form is co-exact.
 
 To take the second route, there is a generalization of Theorem 3.6 of \cite{GaNe22} in the case $n = 2$:

\begin{theorem} 
\label{diffFh}
Let $v$ be a continuous function such that $v|_T \in C^2(T)$ for all $T \subseteq M$, let and $\{g(t) : t \in [0,1]\}$ be a family of Regge metrics such that $t \mapsto g(t)|_T$ is smooth for each triangle $T$. Then,
\[
\pdiff{}{t} [\iprod{\iprod{{\alpha|_{\partial M}}_\mathrm{dist}(g(t)),v}} - \iprod{\iprod{KdA_\mathrm{dist}(g(t)),v}}] = -\frac{1}{2}b_h(g(t);\dot g(t),v).
\]

$b_h$ is defined as
\begin{equation}\label{bheqn}
b_h(g;\sigma,v) := \sum_{T \subseteq M} \iprod{\mathbb{S}\sigma,\nabla\nabla v}_{L^2(T,g)} - \sum_{e \subset M} \iprod{(\mathbb{S}\sigma)(n,n),dv([\![n]\!])}_{L^2(e,g)},
\end{equation}
where $\mathbb{S}\sigma = \sigma - \mathrm{tr}(\sigma)g$ and $n$ denotes the unit normal to $e$ with respect to $g$.
\end{theorem}
Note that the definition of $b_h$ only depends on the mesh $\mathcal{T}_h$ and the metric $g$. Note also that we abuse notation by using the letter $n$ for the normal vector rather than the dimension (which in this context is fixed at $2$).

\begin{proof}[Proof of Theorem \ref{diffFh}]
Since in our case $g = g(t)$, we will use a subscript $t$ to emphasize the dependence on the $g(t)$ metric.  Therefore
\begin{align*}
\pdiff{}{t} \iprod{\iprod{{\alpha|_{\partial M}}_\mathrm{dist}(g(t)),v}} &= \pdiff{}{t}\left[-\sum_{e \subset \partial M} \int_e v k_tds_t - \sum_{p \in \partial M}v(p)\Bigg(\pi - \sum_{T \ni p} \theta_T(p)\Bigg)\right] \\
&= - \sum_{e \subset \partial M} \int_e v\pdiff{}{t} k_tds_t + \sum_{p \in \partial M} v(p)\sum_{T \ni p} \pdiff{}{t}\theta_T(p) \\
&= - \frac{1}{2}\Bigg(\sum_{e \subset \partial M}\int_ev[d(\dot{g}(n_t,\tau_t))(\tau_t) + (\mathrm{div}(\mathbb{S}\dot g))(n_t)]ds_t \\
&\quad\quad\quad- \sum_{p \in \partial M}\sum_{T \ni p} v(p)[\![\dot{g}(n_t,\tau_t)]\!]|_p \Bigg).
\end{align*}

The identities used here are established in \cite{GaNe22}, Section 2. $\tau_t$ is a tangent unit vector to $e$.

In the second to last line of the proof of Lemma 3.3 of \cite{GaNe22}, the following identity is established (simplified here for the $n = 2$ case):
\begin{align*}
b_h(g(t);\dot g, v) = 2\pdiff{}{t}\iprod{\iprod{(KdA)_\mathrm{dist}(g(t)),v}}  &+ \sum_{p \in \partial M} \sum_{T \ni p}v(p)[\![\dot{g}(n_t,\tau_t)_T]\!]|_p\\
&- \sum_{e \subset \partial M} \int_e v[d(\dot{g}(n_t,\tau_t))(\tau_t) + (\mathrm{div}(\mathbb{S}\dot{g}))(n_t)]ds_t.
\end{align*}
Comparing with the value of $\pdiff{}{t}\iprod{\iprod{{\alpha|_{\partial M}}_\mathrm{dist}(g(t)),v}}$, we get the desired result. 
\end{proof}
Define the functional $F_h(g(t)) \in V_h^*$ by $F_h(g(t))(v) := \iprod{\iprod{{\alpha|_{\partial M}}_\mathrm{dist}(g(t)),v}} - \iprod{\iprod{KdA_{\mathrm{dist}}(g(t)),v}}$. In the framework of geometric variational crimes, $F_h(g)$ corresponds to the functional $F$ in  (\ref{errbounds}) while $F_h(g_h)$ corresponds to the functional $F_h$ on the right of the discrete problem (\ref{discproblem}) and in (\ref{errbounds}). Therefore a key component of the error analysis will be bounds on $\|F_h(g) - F_h(g_h)\|_{V_h^*}$.

\section{Error Analysis of the Connection Form Problem}
The purpose of this section is to establish asymptotic error bounds on \\$\|\star_h du_h - \star du\|_{L^2(M,g)}$ in terms of the metric error. We have already established in Theorem \ref{diffFh} that $\pdiff{}{t}F_h(g(t);v_h) = -\frac{1}{2}b_h(g(t);\dot g(t), v_h)$. In particular we can set $g(t) = tg_h + (1-t)g$, so that \\$\|F_h(g;\cdot) - F_h(g_h;\cdot)\|_{V_h^*} \le \sup_{v_h \in V_h}\frac{\int_0^1 \frac{1}{2}|b_h(g(t);g_h - g,v_h)| dt}{\|v_h\|_{V_h}} \le \sup_{t \in [0,1]} \frac{1}{2}\|b_h(g(t);g_h - g,\cdot)\|_{V_h^*}$.

\begin{theorem}\label{Fhbounds}
Let $(M,g)$ be a compact disc with Riemannian metric $g$, and let $\{(M,\mathcal{T}_h,g_h)\}_{h \in S}$ be a Riemannian shape-regular and quasi-uniform set of meshes and Regge metrics, and let $V_h \subset H(d)\Omega^0(M,g_h)$ be a space such that for each $v_h \in V_h$ and each triangle $T \subseteq M$, $v_h|_T = f_{h,T}^{-*}(\hat v)$ for some $\hat v \in \hat V \subset C^2\Omega^0(\hat T)$, where $\hat V$ is a fixed finite-dimensional subspace and $\hat T$ is the standard simplex. Assume additionally that there exists an integer $r \ge 0$ such that
\begin{enumerate}\item{$\|g_h - g\|_{L^2(M,g)} = O(h^{r+1})$,}
\item{$\left(\sum_{T \subseteq M} \|\nabla g_h\|_{L^2(T,g)}^2\right)^\frac{1}{2} = O(h^r)$,}
\item{$\|g_h - g\|_{L^\infty(M,g)} \le C' < 1$ for all $h \in S$,}
\item{$\max_{T \subseteq M}\|\hat \nabla g_h\|_{L^\infty(T,\hat \delta)} + \|\hat \nabla g\|_{L^\infty(T,\hat \delta)} = O(1)$, where $\hat \delta$ is the flat metric induced by the map $f_{h,T}: \hat T \to T$ and $\hat{\nabla}$ is its associated Levi-Civita connection.} 
\end{enumerate}
Then, treating values that depend only on $(M,g)$, $C'$, and the mesh quality bounds $K$, $K_V$, and $K'$ appearing in the definitions of shape-regularity and quasi-uniformity (see Section \ref{shapereg_section}) as constants, we have
\[
\|b_h(g(t);g - g_h,\cdot)\|_{V_h^*} \in O(h^r),
\]
where $b_h$ is as in equation (\ref{bheqn}).
\end{theorem}
\begin{proof}
We'll bound each term of $b_h$ separately. For clarity, we will elide values which are bounded uniformly by a constant that does not depend on $h$.

A few things need to be established to simplify the bounds. We'll show that $C_{g,g(t)}(M)$ and $C_{g(t),g}(M)$ both converge to 1 as fast as $\|g - g_h\|_{L^\infty(M,g)}$ converges to zero. 

Since for any two Regge metrics $g_1, g_2$ which are both smooth in a neighborhood of $x \in M$,
\begin{align*}
    \sup_{V \in T_xM \backslash \{0\}} \frac{\|V\|^2_{g_1}}{\|V\|^2_{g_2}} &\le \sup_{V,W \in T_xM \backslash \{0\}} \frac{\iprod{V,W}_{g_1}}{\|V\|_{g_2}\|W\|_{g_2}} \\ &= \sup_{V,W \in T_xM \backslash \{0\}} \frac{\iprod{V,W}_{g_1} - \iprod{V,W}_{g_2}}{\|V\|_{g_2}\|W\|_{g_2}} + \frac{\iprod{V,W}_{g_2}}{\|V\|_{g_2}\|W\|_{g_2}}\\
&\le \|g_1 - g_2\|_{L^\infty(M,g_2)} + 1,
\end{align*}
we have $C_{g(t),g}(M)^2 \le \|g(t) - g\|_{L^\infty(M,g)} + 1$ and $C_{g,g(t)}(M)^2 \le \|g(t) - g\|_{L^\infty(M,g(t))} + 1$. Clearly, $\|g(t) - g\| \le \|g_h - g\|$ in any metric. To make use of the second inequality, we need to do some work: $C_{g,g(t)}(M)^2 \le \|g_h - g\|_{L^\infty(M,g(t))} + 1\le C_{g,g(t)}(M)^2\|g_h - g\|_{L^\infty(M,g)} + 1$, so we obtain
$C_{g,g(t)}(M)^2 = 1 + O(\|g_h - g\|_{L^\infty(M,g)})$. So, any instance of $C_{g(t),g}(T)$, $C_{g(t),g}(M)$, $C_{g,g(t)}(T)$, or $C_{g,g(t)}(M)$ will be suppressed. An identical argument allows us to suppress terms like $C_{g(t),g_h}$ and $C_{g_h,g(t)}$, as well as terms involving $D$ (by part 1 of Lemma \ref{L2_equiv}).

Another technicality that needs to be addressed: if $\{(M,\mathcal{T}_h,g_h)\}_{h \in S}$ is Riemannian shape-regular and quasi-uniform, then so is $\{(M,\mathcal{T}_h,tg + (1-t)g_h)\}_{h \in S}$ for any $t \in [0,1]$, with constants $K(t)$, $K_V(t)$, and $K'(t)$ bounded by a multiple of those for $g_h$. This is because $C_{\hat \delta,g(t)}(T) \le C_{\hat \delta, g_h}(T)C_{g_h,g(t)}(T)$ and $C_{g(t),\hat \delta}(T) \le C_{g(t),g_h}(T)C_{g_h,\hat \delta}(T)$.

Now let us begin bounding the terms of $b_h$. In what follows, we use the letter $C$ to denote a constant that is independent of $h$ and may change at each occurrence.  We have
\begin{align*}|\iprod{\mathbb{S}_t\sigma,\nabla_t\nabla_t & v_h}_{L^2(T,g(t))}| \le \|\mathbb{S}_t\sigma\|_{L^2(T,g(t))}\|\nabla_t\nabla_t v_h\|_{L^2(T,g(t))}\\
&\le C \|\sigma\|_{L^2(T,g(t))}h^{-1}\left(1 + \|\hat \nabla g(t)\|_{L^\infty(T,\hat \delta)}\right)\|dv_h\|_{L^2(T,g(t))} \\
&\le C\|g - g_h\|_{L^2(T,g)}h^{-1}\left(1 + \|\hat \nabla g_h\|_{L^\infty(T,\hat \delta)} + \|\hat \nabla g\|_{L^\infty(T,\hat \delta)}\right)\|dv_h\|_{L^2(T,g_h)}.
\end{align*}
In the above derivation: the first line is an application of Cauchy-Schwarz, the second is applying the inverse inequality of corollary \ref{shapereginverse}, and the last line is changing from $g(t)$ to $g$ or $g_h$ norms. Recall that $\mathbb{S}_t\sigma(V,W) = \sigma(V,W) - (\mathrm{Tr}_t\sigma)\iprod{V,W}_t$, and $|\mathrm{Tr}_t\sigma| \le 2 \sup_{V,W \in T_x(T)} \frac{\sigma(V,W)}{\|V\|_t\|W\|_t}$, so $\|\mathbb{S}_t\sigma\|_{L^2(T,g(t))} \le 3\|\sigma\|_{L^2(T,g(t))}$.

Next are the edge terms. Let $n_\alpha$ and $n_\beta$ be the two outward-facing normal vectors of $e$ for the two triangles $T_\alpha \supset e$ and $T_\beta \supset e$ in $g(t)|_{T_\alpha}$ and $g(t)|_{T_\beta}$ respectively. If $e \subset \partial M$, then we will say that $n_\beta = 0$ and the terms involving the non-existent $T_\beta$ can be ignored.  First we unravel the definitions:
\begin{align}
|\iprod{&(\mathbb{S}_t\sigma)(n_t,n_t),dv_h([\![n_t]\!]}_{L^2(e,g(t))}| = |\iprod{\sigma(\tau_t,\tau_t),dv_h([\![n_t]\!])}_{L^2(e,g(t))}|\notag \\
&\le \|\sigma\|_{L^2(e,g(t))}\left(\|dv_h(n_\alpha)\|_{L^2(e,g(t))} + \|dv_h(n_\beta)\|_{L^2(e,g(t))}\right) \notag \\
&\le \|\sigma\|_{L^2(e,g)}\left(\|dv_h\|_{L^2(e,g_h)} + \|dv_h\|_{L^2(e,g_h)}\right). \label{edgeeqn}\end{align}

Here we applied the Cauchy-Schwarz inequality and then converted from $g(t)$ norms to $g$ and $g_h$ norms. By a standard application of the scaled trace and inverse inequalities (\ref{shaperegtrace}) and (\ref{shapereginverse}) to each coefficient of $dv_h$ in the coordinates $(\hat{x},\hat{y})$ induced by the diffeomorphism $\hat{T} \to T$, we get that for each edge $e$ and each $T_\alpha \supset e$,

\[\|dv_h\|_{L^2(e,g_h)} \le C h^{-\frac{1}{2}}\|dv_h\|_{L^2(T_\alpha,g_h)}.\]

Plugging this into (\ref{edgeeqn}), we get
\begin{align*}
(\ref{edgeeqn}) &\le C\|\sigma\|_{L^2(e,g)}h^{-\frac{1}{2}}\left(\|dv_h\|_{L^2(T_\alpha,g_h)} + \|dv_h\|_{L^2(T_\beta,g_h)}\right) \\
&\le Ch^{-\frac{1}{2}}\left(h^{-\frac{1}{2}}(1 + \|\hat \nabla g\|_{L^\infty(T_\alpha,\hat \delta)})\|\sigma\|_{L^2(T_\alpha,g)} + h^\frac{1}{2}\|\nabla \sigma\|_{L^2(T_\alpha,g)}\right)\\
&\quad\quad\quad\quad\left(\|dv_h\|_{L^2(T_\alpha,g_h)} + \|dv_h\|_{L^2(T_\beta,g_h)}\right)\\
&= C\left(h^{-1}\left(1 + \|\hat \nabla g\|_{L^\infty(T,\hat \delta)}\right)\|g - g_h\|_{L^2(T_\alpha,g_h)} + \|\nabla g_h\|_{L^2(T_\alpha,g_h)}\right)\\
&\quad\quad\quad\quad\left(\|dv_h\|_{L^2(T_\alpha,g_h)} + \|dv_h\|_{L^2(T_\beta,g_h)}\right).
\end{align*}

Here, we applied the trace inequality of Corollary \ref{shaperegtrace} to the $\sigma$ terms, then rearranged the $h$ coefficients into a convenient form, then expanded the definition of $\sigma$. Note that $\nabla g = 0$.

From these two bounds on the different terms appearing in $b_h$, it is clear that 
\begin{align*}
|b_h(g(t);g-g_h,v_h)| &\le C\|dv_h\|_{L^2(M,g_h)}\Bigg[\Bigg(\sum_{T \subset M}\|\nabla g_h\|_{L^2(T,g)}^2\Bigg)^\frac{1}{2}\\
&\quad\quad + h^{-1}\|g - g_h\|_{L^2(M,g)}\left(2 + \max_{T \subseteq M} \|\hat \nabla g_h\|_{L^\infty(T,\hat \delta)} + \|\hat \nabla g\|_{L^\infty(T,\hat \delta)}\right) \Bigg]\\
&\le C \|dv_h\|_{L^2(M,g_h)} h^r,
\end{align*}
where $C$ depends only on $(M,g)$, the constant $C' < 1$ which bounds $\|g - g_h\|_{L^\infty(M,g)}$, and the mesh quality measures $K$, $K_V$, and $K'$.
\end{proof}

Since $F_h(g;v_h) - F_h(g_h;v_h) = \frac{1}{2}\int_0^1 b_h(tg + (1 - t)g_h;g-g_h,v_h) dt$, this immediately implies that $\|F_h(g;\cdot) - F_h(g_h;\cdot)\|_{V_h^*} = O(h^r)$.

With this, we can conclude that if we have a finite element space $V_h \subset H(d)\Omega^0(M,g_h)$ which satisfies the conditions of Theorem \ref{Fhbounds}, and which is contained in a sequence of spaces admitting bounded cochain projections $\pi_h^k$, and which contains all constant functions and satisfies $\inf_{v \in V_h}\|u - v\|_{H(d)(M,g)} = O(h^r)$, and we have an interpolant $g_h$ satisfying the hypotheses of Theorem \ref{Fhbounds} and $\|g - g_h\|_{L^\infty(M,g)} = O(h^r)$, then $\|du_h - du\|_{L^2(M,g)} = O(h^r)$ and hence $\|\star_h du_h - \star du\|_{L^2(M,g)} = O(h^r)$.

In the case that $M$ has the topology of a 2-dimensional disc, it is practically sufficient to consider a shape-regular, quasi-uniform mesh $\mathcal{T}_h$ on a closed convex domain $\Omega \subset \mathbb{R}^2$ and construct a finite element space $V_h$ using this mesh, such as Lagrange elements of order $r$. Additionally, setting $\bar \delta$ to the Euclidean metric induced by this coordinate choice, we have $\hat \nabla = \bar \nabla$ since $f_{h,T}$ is affine when expressed in this coordinate system, so $\|\hat \nabla g\|_{L^\infty(T,\hat \delta)} \le C_{\bar \delta,\hat \delta}(T)^3 \|\bar \nabla g\|_{L^\infty(T,\bar \delta)} \le Ch^3\|\bar \nabla g\|_{L^\infty(T,\bar \delta)}$ if mesh elements have Euclidean diameter $O(h)$, and likewise for $g_h$. Order-$r$ interpolants $g_h$ for $g$ satisfying the other hypotheses of Theorem \ref{Fhbounds} also exist, since interpolants with the correct convergence properties in Euclidean space exist \cite{Li,GaNe22} and the Euclidean metric is obviously related to the $g$ metric in a bounded way.

\correction{}{More generally, the Regge finite elements~\cite{Li} can be used to construct $g_h$ on an $n$-dimensional manifold $M$ that is not necessarily an open subset of Euclidean space. The restriction of $g_h$ to each $T \subseteq M$ is taken to be $f_{h,T}^{-*} \hat{g}_T$ for some symmetric tensor field $\hat{g}_T$ on the standard simplex $\hat{T}$ that has polynomial coefficients.  The degrees of freedom for $g_h|_T$ are suitable moments of $\hat{g}_T$ on faces of $\hat{T}$, as detailed in~\cite{Li}.  The tangential-tangential continuity of $g_h$ is enforced by equating degrees of freedom associated with shared interfaces in the same way that one does on Euclidean domains.  In more detail, let $e = T_1 \cap T_2$ be a codimension-1 interface between two simplices $T_1,T_2 \subseteq M$, and let $\hat{e}_j = f_{h,T_j}^{-1}(e)$ for $j=1,2$, so that the map $f_{h,T_2}^{-1}\big|_e \circ f_{h,T_1}\big|_{\hat{e}_1}$ sends $\hat{e}_1$ to $\hat{e}_2$.  If we assume that $f_{h,T_2}^{-1}\big|_e \circ f_{h,T_1}\big|_{\hat{e}_1}$ is affine and we equate the degrees of freedom associated with $e \subset T_1$ and $e \subset T_2$, then we obtain
\[
i_{\hat{e}_1}^* \hat{g}_{T_1} = \left( f_{h,T_2}^{-1}\big|_e \circ f_{h,T_1}\big|_{\hat{e}_1} \right)^* i_{\hat{e}_2}^* \hat{g}_{T_2},
\]
where $i_{\hat{e}_j}$ denotes the inclusion $\hat{e}_j \hookrightarrow \hat{T}$.  Since $\hat{g}_{T_j} = f_{h,T_j}^* (g_h|_{T_j})$ and the map $f_{h,T_j} \circ i_{\hat{e}_j} \circ f_{h,T_j}^{-1}\big|_e$ is simply the inclusion $i_{e,T_j} : e \hookrightarrow T_j$, this implies that
\[
i_{e,T_1}^* (g_h|_{T_1}) = i_{e,T_2}^* (g_h|_{T_2}).
\]
In other words, the tangential-tangential components of $g_h$ are single-valued on $e$.  Note that this is an equality that holds pointwise on a curved facet $e \subset M$ whose tangent vectors vary with position.  Note also that the reasoning above required that $f_{h,T_2}^{-1}\big|_e \circ f_{h,T_1}\big|_{\hat{e}_1}$ is affine.  This is a mild hypothesis that is satisfied, for instance, in the setting where $M$ is a smooth hypersurface in $\mathbb{R}^{n+1}$ and $f_{h,T_j}$ is the composition of two maps: an affine map that sends the standard simplex to a flat simplex $\widetilde{T}_j$ with vertices on $M$, and a nonlinear map that sends points on $\widetilde{T}_j$ to their closest point projection onto $M$.
}

\section{Numerical Example and Benchmarking}

The discrete problem presented in Section 5 was implemented in python using the NGSolve framework~\cite{schoeberl1997netgen,schoeberl2014ngsolve}. In order to benchmark against an analytically solved example, such an example needed to be calculated. We chose a spherical cap of radius $1$ with the parameterization $(x,y) \mapsto (x,y,\sqrt{1 - x^2  - y^2})$ because it is easy to calculate the Gauss curvature, and we were lucky to guess the correct frame:

\[e_1 = \sqrt{\frac{1 - x^2 - y^2}{1 - y^2}}\pdiff{}{x},\]
\[e_2 = \frac{-xy}{\sqrt{1 - y^2}}\pdiff{}{x} + \sqrt{1 - y^2}\pdiff{}{y},\]
\[\alpha = -\frac{1}{\sqrt{1 - x^2 - y^2}}\left(ydx + \frac{xy^2}{1 - y^2}dy\right).\]

It can be checked (with much labor) that $d\alpha = \frac{1}{\sqrt{1 - x^2 - y^2}} dx\wedge dy = K dA$, $d\star\alpha = 0$, and $\alpha = \iprod{\nabla e_2,e_1}$.

Our benchmarks were evaluated on the domain $(x,y) \in [-\frac{1}{4},\frac{1}{4}]^2$. Meshes were created with the ``GenerateMesh" function, generating unstructured meshes with (Euclidean) edge length approximately equal to $h$. Each interior vertex was then perturbed with uniform randomness in the range $[-\frac{h}{4},\frac{h}{4}]^2$. The order-$r$ Regge finite element space~\cite{Li} and interpolant described in~\cite[Appendix A]{GaNe22} was used for $g_h$, and the order-$r$ Lagrange finite element space was used for $V_h$.  \correction{}{The aforementioned interpolant onto the Regge finite element space is constructed as follows: One first computes an elementwise $L^2$-projection of the metric onto the (globally discontinuous) space of piecewise polynomial symmetric tensor fields of degree at most $r$.  Then, as explained in~\cite[Appendix A]{GaNe22}, degrees of freedom on shared interfaces are averaged to produce a member of the order-$r$ Regge finite element space.  Experiments were performed with $r \in \{0,1,2,3\}$.} The python code used to produce this data can be provided upon request.

As one can see from Figure~\ref{fig:convergence}, the numerical scheme's convergence in the $L^2(M,g)$-norm very closely matches a priori predictions.  \correction{}{Table~\ref{tab:convergence} displays the same data in tabular form.  In the table, ``ndof'' refers to the dimension of the order-$r$ Regge finite element space plus the the dimension of the order-$r$ Lagrange finite element space.}

\begin{figure}[htp]
\centering
\includegraphics[scale=0.6,trim=0in 0.5in 0in 0in,clip=true]{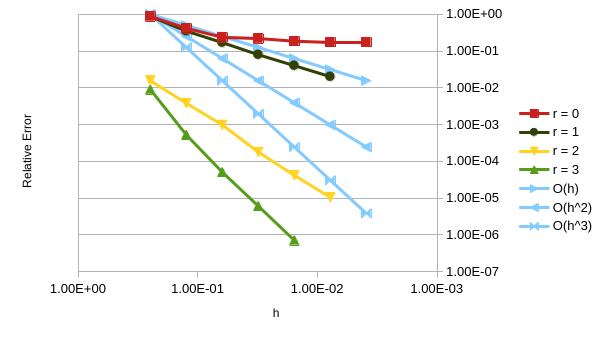}
\caption{Graphs of the relative error $\|du_h - \star\alpha\|_{L^2(M,g)}/\|\alpha\|_{L^2(M,g)}$ vs. the parameter $h$ which controls the diameter of mesh elements.}
\label{fig:convergence}
\end{figure}

\begin{figure}[htp]
\centering
\begin{tabular}{|l|l|l|l|l|l|l|l|l|l|l|}
\hline
	  & \multicolumn{2}{c|}{$r = 0$} & \multicolumn{2}{c|}{$r = 1$}\\
	\hline
	$h$ & $E$ & ndof & $E$ & ndof \\
	\hline
	2.5E-1 & 8.75E-1 &	22  &   8.73E-1 &	66 \\
1.3E-1 & 4.16E-1 &	86 &	3.48E-1 &	306 \\
6.3E-2 & 2.36E-1 &	310 &	1.69E-1 &	1170 \\
3.1E-2 & 2.18E-1 &	1286 &	7.93E-2 &	5010\\
1.6E-2 & 1.85E-1 &	4906&	4.01E-2 &	19362 \\
7.8E-3 & 1.73E-1 &	19334 &	2.00E-2 &	76818 \\
3.9E-3 & 1.74E-1 &	76274 & \multicolumn{2}{c|}{}\\
\hline
\end{tabular}
\begin{tabular}{|l|l|l|l|l|}
\hline
 & \multicolumn{2}{c|}{$r = 2$} & \multicolumn{2}{c|}{$r = 3$} \\
\hline
$h$ & $E$ & ndof & $E$ & ndof  \\
 \hline
 2.5E-1 & 1.56E-2 &	134 &	9.00E-3 &	226 \\
1.3E-1 & 3.80E-3 &	662 &	5.24E-4 &	1154 \\
6.3E-2 & 9.75E-4 &	2582 &	5.06E-5 &	4546 \\
3.1E-2 & 1.78E-4 &	11174 &	6.06E-6 &	19778  \\
1.6E-2 & 4.15E-5 &	43370 &	7.09E-7 &	76930 \\
7.8E-3 & 1.05E-5 &	172454 & \multicolumn{2}{c|}{} \\		
\hline
\end{tabular}
\caption{Table containing $h$ = average triangle diameter, $E$ = relative error of $du_h$ measured in the $L^2(M,g)$ norm, ndof = total degrees of freedom (including those used in constructing the Regge metric).}
\label{tab:convergence}
\end{figure}
\pagebreak

\section{Appendix}
The following small lemma is very useful for showing results related to $C_{g_1,g_2}(\Omega)$.
\begin{lemma}\label{norm_comparison_equiv}
Let $V$ be a vector space and $(V,\|\cdot\|_{g_1})$, $(V,\|\cdot\|_{g_2})$ be two reflexive Banach space structures on $V$. By abuse of notation, also use the same symbols for the induced operator norms on $V^*$. Then:
\begin{equation}
\sup_{\alpha \in V^*, \;\alpha \ne 0} \frac{\|\alpha\|_{g_1}}{\|\alpha\|_{g_2}} = \sup_{v \in V,\;v \ne 0} \frac{\|v\|_{g_2}}{\|v\|_{g_1}} \end{equation}
and
\begin{equation}\inf_{v \in V,\; v \ne 0} \frac{\|v\|_{g_2}}{\|v\|_{g_1}} = \inf_{\alpha \in V^*, \;\alpha \ne 0} \frac{\|\alpha\|_{g_1}}{\|\alpha\|_{g_2}}. \end{equation}
\end{lemma}

\begin{proof}

Let $\alpha \in V^*$. Then:
\begin{align*}\|\alpha\|_{g_1} &= \sup_{v \ne 0} \frac{|\alpha(v)|}{\|v\|_{g_1}} = \sup_{v \ne 0} \frac{\left|\alpha\left(\frac{v}{\|v\|_{g_2}}\right)\right|}{\frac{\|v\|_{g_1}}{\|v\|_{g_2}}}\\
&\le \left(\sup_{v \ne 0} \frac{\|v\|_{g_2}}{\|v\|_{g_1}}\right)\left(\sup_{v \ne 0} \frac{|\alpha(v)|}{\|v\|_{g_2}}\right) = \sup_{v \in V,\; v \ne 0} \frac{\|v\|_{g_2}}{\|v\|_{g_1}}\|\alpha\|_{g_2}.\end{align*}
Since the operator norms on $V^{**}$ are isometric to the original norms on $V$, the same reasoning also shows that if $v \in V = V^{**}$ then $\|v\|_{g_2} \le \sup_{\alpha \in V^*,\;\alpha \ne 0}\frac{\|\alpha\|_{g_1}}{\|\alpha\|_{g_2}}\|v\|_{g_1}$, proving equality in the first statement.

The second statement follows from the fact that $\inf_x \frac{a(x)}{b(x)} = \frac{1}{\sup_x \frac{b(x)}{a(x)}}$.
\end{proof}

\begin{proof}[Proof of Theorem \ref{L2_equiv}]
We will establish a uniform bound on $\|\alpha\|_{g_1}$ when $\alpha \in T^*_x(M)^p \otimes\Lambda^k_x(M),\; \|\alpha\|_{g_2} = 1$, with a special case for $k = n,\; p = 0$ (in which case $\alpha = dV_{g_2}$). In this proof, $x \in M$ is always a point around which $g_1$ and $g_2$ are both smooth in a neighborhood. We will also use $\{F_i\}_{i = 1}^n$ and $\{\theta^i\}_{i = 1}^n$ to refer to a frame and coframe that are orthonormal in the metric $g_1$ and dual to each other, and $\{\tilde F_i\}_{i = 1}^n$ and $\{\tilde \theta^i\}_{i = 1}^n$ to refer to a frame and coframe that are orthonormal in the metric $g_2$ and dual to each other. The two frames will also induce the same orientation on $M$.

Since the covector norm is the same as the operator norm, \\$\|dV_{g_2}\|_{g_1} = \|\tilde\theta^1 \wedge \dots \wedge \tilde\theta^n\|_{g_1} = \det{[\tilde\theta^i(F_j)]}$. Let $\tilde\Theta$ be the matrix defined by $\tilde\Theta_{ij} = \tilde \theta^i(F_j)$. If $u,w \in \mathbb{R}^n$, then $u \cdot \tilde\Theta w = \iprod{u^i\tilde F_i,w^jF_j}_{g_2}$. The eigenvalues of the matrix $\tilde \Theta$ have magnitude bounded above by $\|\tilde\Theta\|_2 = \sup_{w \in \mathbb{R}^n\backslash \{0\}} \frac{\|\tilde\Theta w\|_2}{\|w\|_2} = \sup_{u,w \in \mathbb{R}^n \backslash \{0\}} \frac{u \cdot\tilde\Theta w}{{\|u\|_2\|w\|_2}}$. However, writing $u' = u^i\tilde F_i$ and $w' = w^jF_j$, this expression is the same as $\sup_{u',w' \in T_x(M) \backslash \{0\}} \frac{\iprod{u',w'}_{g_2}}{\|u'\|_{g_2}\|w'\|_{g_1}} = \sup_{w' \in T_x(M) \backslash \{0\}} \frac{\|w'\|_{g_2}}{\|w'\|_{g_1}}$. Likewise, the eigenvalues of $\tilde\Theta$ have magnitude bounded below by $\inf_{w \in \mathbb{R}^n \backslash \{0\}} \frac{\|\tilde\Theta w\|_2}{\|w\|_2} = \inf_{w' \in T_x(M)\backslash\{0\}}\frac{\|w'\|_{g_2}}{\|w'\|_{g_1}}$. The determinant is the product of eigenvalues, so we have
\begin{align*}\frac{1}{C_{g_1,g_2}(M)^n} &\le \left(\frac{1}{\sup_{v \in T_x(M),\; v\ne 0}\frac{\|v\|_{g_1}}{\|v\|_{g_2}}}\right)^n \\
 &=\left(\inf_{v \in T_x(M),\; v\ne 0}\frac{\|v\|_{g_2}}{\|v\|_{g_1}}\right)^n \le \|dV_{g_2}\|_{g_1} \le \left(\sup_{v \in T_x(M),\;v \ne 0} \frac{\|v\|_{g_2}}{\|v\|_{g_1}}\right)^n = C_{g_2,g_1}(M)^n,\end{align*}
 proving part 1 of the theorem.

To bound $\|\alpha\|_{g_1}$ for general $p$ and $k$, we express $\alpha$ in coordinates: \\$\alpha = \sum_{I,J} \alpha_{IJ} (\tilde\theta^{j_1} \otimes \dots \otimes \tilde\theta^{j_p}) \otimes \tilde\theta^{i_1}\wedge\dots\wedge\tilde\theta^{i_k}$, where $I$ ranges over all increasing multi-indices $(i_1,\dots,i_k) \in [1,n]^k$ and $J$ ranges over all tuples of indices $(j_1,\dots,j_p) \in [1,n]^p$. Since $\|\alpha\|_{g_2} = 1$, we have $\sum_{I,J} \alpha_{IJ}^2 = 1$. For convenience, in the following lines we will use $\tilde\theta^I$ to refer to $\tilde\theta^{i_1}\wedge \dots \wedge\tilde\theta^{i_k}$ and $\otimes^J\tilde\theta$ to refer to $\tilde\theta^{j_1}\otimes \dots \otimes \tilde\theta^{j_p}$.  We have
\begin{align*}
\iprod{\alpha,\alpha}_{g_1} &= \sum_{I,J,L,M} \alpha_{IJ}\alpha_{LM} \iprod{\otimes^J\tilde\theta,\otimes^M\tilde\theta}_{g_1}\iprod{\tilde\theta^I,\tilde\theta^L}_{g_1}\\
&= \sum_{I,J,L,M} \alpha_{IJ}\alpha_{LM}\prod_{s = 1}^p \iprod{\tilde\theta^{j_s},\tilde\theta^{m_s}}_{g_1}\sum_{\sigma \in S_{k}}\mathrm{sign}(\sigma)\prod_{t = 1}^{k} \iprod{ \tilde\theta^{i_t}, \tilde\theta^{l_{\sigma(t)}} }_{g_1}.
\end{align*}

The inner product $\iprod{\tilde\theta^I,\tilde \theta^L}_{g_1}$ above is equal to the determinant of the $k\times k$ matrix $\tilde\Theta^{IL}_{tr} = \iprod{\tilde\theta^{i_t},\tilde\theta^{l_r}}_{g_1}$. Since $\tilde \theta^i = \tilde\theta^i(F_a)\theta^a$, we have $\tilde\Theta^{IL} = E^I(E^L)^\intercal$, where $E^I_{ta} = \tilde\theta^{i_t}(F_a)$ and likewise for $E^L$. As before, for $u \in \mathbb{R}^k$ and $w \in \mathbb{R}^n$, we have $u \cdot E^I w = \iprod{u^t \tilde F_{i_t},w^a F_a}_{g_2}$, and $\|E^I\|_2 = \sup_{u' \in \mathrm{span}(\tilde{F}_{i_1},\dots,\tilde{F}_{i_k}) \backslash \{0\}, w' \in T_x(M)\backslash \{0\}} \frac{\iprod{u',w'}_{g_2}}{\|u'\|_{g_2}\|w'\|_{g_1}}$. By expanding the range of $u'$, we get $\|E^I\|_2 \le \sup_{w' \in T_x(M) \backslash \{0\}} \frac{\|w'\|_{g_2}}{\|w'\|_{g_1}}$. Similarly, $w \cdot (E^L)^\intercal u = \iprod{w^b F_b,u^r\tilde F_{l_r}}_{g_1}$, so $\|(E^L)^\intercal\|_2 \le \sup_{u' \in T_x(M) \backslash \{0\}} \frac{\|u'\|_{g_2}}{\|u'\|_{g_1}}$. Therefore 
\[
|\iprod{\tilde \theta^I,\tilde \theta^L}_{g_1}| = |\mathrm{det}(\tilde\Theta^{IL})| \le \|\tilde \Theta^{IL}\|_2^k \le \|E^I\|_2^k\|(E^L)^\intercal\|_2^k \le \sup_{v \in T_x(M) \backslash \{0\}} \left(\frac{\|v\|_{g_2}}{\|v\|_{g_1}}\right)^{2k}.
\]
The terms $\iprod{\tilde \theta^{j_s},\tilde\theta^{m_s}}_{g_1}$, meanwhile, are simply bounded above by $\sup_{\tilde \theta \in T_x^*(M) \backslash \{0\}} \frac{\|\tilde \theta\|_{g_1}^2}{\|\tilde \theta\|_{g_2}^2}$ by the Cauchy-Schwarz inequality.

Therefore
\begin{align*}
|\iprod{\alpha,\alpha}_{g_1}| &\le \sup_{\tilde \theta \in T_x^*(M) \backslash \{0\}} \left(\frac{\|\tilde \theta\|_{g_1}}{\|\tilde \theta\|_{g_2}}\right)^{2p}\sup_{v \in T_x(M)\backslash \{0\}}\left(\frac{\|v\|_{g_2}}{\|v\|_{g_1}}\right)^{2k}\sum_{I,J,L,M} |\alpha_{IJ}\alpha_{LM}| \\
&= \sup_{v \in T_x(M) \backslash \{0\}} \left(\frac{\|v\|_{g_2}}{\|v\|_{g_1}}\right)^{2k + 2p} \sum_{I,J}|\alpha_{IJ}|\sum_{L,M}|\alpha_{LM}| \\
&\le C_{g_2,g_1}(M)^{2(k + p)} {n \choose k}n^p,
\end{align*}
since $\sqrt{{n \choose k}n^p}$ is the maximum 1-norm of all unit vectors in $\mathbb{R}^{{n\choose k}n^p}$. Thus, for all $\alpha \in T^*_x(M)^p \otimes\Lambda^k_x(M)$, $\|\alpha\|_{g_1}^2 \le {n \choose k}n^pC_{g_2,g_1}(M)^{2(k + p)}\|\alpha\|_{g_2}^2$.

Part 2 of the theorem is now almost immediate. Let $\alpha \in C^\infty_c\Omega^{p,k}(M)$. Then 
\begin{align*}\|\alpha\|_{L^2(M,g_1)}^2 &= \sum_{T \in \mathcal{T}_1} \int_T \|\alpha\|_{g_1}^2dV_{g_1} \\
&\le {n \choose k}n^p\sum_{T \in \mathcal{T}_1, \; T' \in \mathcal{T}_2}C_{g_2,g_1}(T \cap T')^{2k + 2p} \int_{T \cap T'} \|\alpha\|_{g_2}^2\|dV_{g_1}\|_{g_2}dV_{g_2} \\
&\le{n \choose k}n^p\sum_{T' \in \mathcal{T}_2}C_{g_2,g_1}(T')^{2k + 2p}D_{g_1,g_2}(T')\int_{T'}\|\alpha\|_{g_2}^2dV_{g_2}\\
&\le{n \choose k}n^pC_{g_2,g_1}(M)^{2k + 2p}D_{g_1,g_2}(M)\|\alpha\|_{L^2(M,g_2)}^2.
\end{align*}

Since $C^\infty_c\Omega^{p,k}(M)$ is dense in $L^2\Omega^{p,k}(M,g_2)$, the same inequality holds for all $\alpha \in L^2\Omega^{p,k}(M,g_2)$. 
\end{proof}

\begin{proof}[Proof of Lemma \ref{compare_cov}]
To see that $\nabla - \nabla'$ is a bundle morphism for all $x \in T \cap T'$, all that is necessary is to calculate 
\[
(\nabla - \nabla')_V(fW) = df(V)W + f\nabla_V W - df(V)W - f\nabla'_VW = f(\nabla - \nabla')_VW.
\]
What this shows is that $\nabla - \nabla'$ is $C^\infty(T)$-linear, which by definition means $(\nabla - \nabla')|_x: T_x(T \cap T') \otimes T_x(T \cap T') \to T_x(T \cap T')$ is a well-defined linear map for all $x \in T \cap T'$. Clearly this means $(\nabla - \nabla')|_x: T_x^*(T \cap T)^p \otimes \Lambda^k_x(T \cap T') \to T^*_x(T \cap T')^{p+1} \otimes \Lambda^k_x(T \cap T')$ is also a well-defined linear map.

Next we will find a bound on $|\Gamma^i_{jk} - \Gamma'^i_{jk}|$, where $\Gamma^i_{jk}$ and $\Gamma'^i_{jk}$ are the Christoffel symbols associated with $g$ and $g'$, respectively. First let $E$ be the Cholesky factor of $G^{-1}$, i.e. $E^{\intercal}E = G^{-1}$, so that $E_i = \sum_j E_{ij} \pdiff{}{x^j}$ is a $g$-orthonormal frame. Then
\begin{align*}
&|\Gamma^i_{jk} - \Gamma'^i_{jk}| = \frac{1}{2}|G^{il}(G_{lj,k} + G_{lk,j} - G_{jk,l}) - G'^{il}(G'_{lj,k} + G'_{lk,j} - G'_{jk,l})|\\
&= \frac{1}{2}|(G^{il} - G'^{il})(G_{lj,k} + G_{lk,j} - G_{jk,l}) + G'^{il}(G_{lj,k} + G_{lk,j} - G_{jk,l} - G'_{lj,k} - G'_{lk,j} + G'_{jk,l})|\\
&\le \frac{3}{2}\Bigg[\bigg(\max_{i}\sum_{j}|G^{-1}_{ij} - G'^{-1}_{ij}|\bigg)\max_{i,j,k}\bigg|dG_{ij}\bigg(\pdiff{}{x^k}\bigg)\bigg| \\
&\quad\quad + \bigg(\max_{i}\sum_j|G'^{-1}_{ij}|\bigg)\max_{i,j,k}\bigg|(dG - dG')_{ij}\bigg(\pdiff{}{x^k}\bigg)\bigg|\Bigg]\\
&\le \frac{3}{2}\Bigg(\|G^{-1} - G'^{-1}\|_\infty\max_{i,j,k}\bigg|dG_{ij}\bigg(\sum_lE^{-1}_{kl}E_l\bigg)\bigg| \\
&\quad\quad + \|G'^{-1}\|_\infty\max_{i,j,k}\bigg|(dG - dG')_{ij}\bigg(\sum_lE^{-1}_{kl}E_l\bigg)\bigg|\Bigg)\\
&\le \frac{3}{2}\|E^{-1}\|_\infty(\|G^{-1} - G'^{-1}\|_\infty\|dG\|_{m,g} + \|G'^{-1}\|_\infty\|dG - dG'\|_{m,g}).
\end{align*}
This allows us to bound $\|(\nabla - \nabla')_{E_i}E_j\|_g$:
\begin{align*}
\|(\nabla - \nabla')_{E_i}E_j|_x\|_g &= \bigg\|\sum_{k,l}E_{ik}E_{jl}(\nabla - \nabla')_\pdiff{}{x^k}\pdiff{}{x^l}\bigg\|_g \\
&= \bigg\|\sum_kE_{ik}\sum_lE_{jl}(\Gamma^o_{kl} - \Gamma'^o_{kl})\pdiff{}{x^o}\bigg\|_g \\
&\le \|E\|_\infty^2\max_{k,l,o}|\Gamma^o_{kl} - \Gamma'^o_{kl}|\bigg\|\sum_p E^{-1}_{op}E_p\bigg\|_g\\
&\le \frac{3}{2}\|E\|_\infty^2\|E^{-1}\|_\infty^2(\|G^{-1} - G'^{-1}\|_\infty\|dG\|_{m,g} + \|G'^{-1}\|_\infty\|dG - dG'\|_{m,g})\\
&\le \frac{3n^\frac{5}{2}}{2}\|E\|^2_2\|E^{-1}\|_2^2(\|G^{-1} - G'^{-1}\|_2\|dG\|_{2,g} + \|G'^{-1}\|_2\|dG - dG'\|_{2,g})\\
&= \frac{3n^\frac{5}{2}}{2}\|G^{-1}\|_2\|G\|_2(\|G^{-1} - G'^{-1}\|_2\|dG\|_{2,g} + \|G'^{-1}\|_2\|dG - dG'\|_{2,g}).
\end{align*}
In the second to last line, we used the fact that $\|\cdot\|_m \le \|\cdot\|_2$ and $\|\cdot\|_\infty \le \sqrt{n}\|\cdot\|_2$ for matrices. In the last line, we used the fact that $\|E\|_2^2 = \|G^{-1}\|_2$ since $E$ is the Cholesky factor of $G^{-1}$.

Then we can calculate
\begin{align*}
\|(\nabla - \nabla')|_x\|_{2,g} &= \sup_{V,W \in T_x(T \cap T') \backslash \{0\}} \frac{\|(\nabla - \nabla')_VW\|_g}{\|V\|_g\|W\|_g} \le n\max_{i,j} \|(\nabla - \nabla')_{E_i}E_j\|_g \\
&\le \frac{3n^\frac{7}{2}}{2}\|G^{-1}\|_2\|G\|_2(\|G^{-1} - G'^{-1}\|_2\|dG\|_{2,g} + \|G'^{-1}\|_2\|dG - dG'\|_{2,g}).
\end{align*}

Lastly, we will calculate the pointwise norm of $(\nabla - \nabla')|_x\alpha$ for an arbitrary $\alpha \in T_x^*(T \cap T')^p \otimes \Lambda^k_x(T \cap T')$:
\begin{align*}
\|&(\nabla - \nabla')|_x\alpha\|_g = \sup_{V,W_1,\dots,W_{k+p} \in T_x(T \cap T') \backslash \{0\}}\frac{|\sum_j\alpha (W_1,\dots,(\nabla - \nabla')_VW_j,\dots,W_{k+p})|}{\|V\|_g\|W_1\|_g\dots\|W_{k+p}\|_g}\\
&\le \sum_j \sup_{V,W_1,\dots,W_{k+p}}\frac{\|(\nabla - \nabla')_VW_j\|_g}{\|V\|_g\|W_j\|_g}\frac{\big|\alpha(W_1,\dots,W_{j-1},\frac{(\nabla - \nabla')_VW_j}{\|(\nabla - \nabla')_VW_j\|_g},W_{j+1},\dots,W_{k+p})\big|}{\|W_1\|_g\dots\widehat{\|W_j\|_g}\dots\|W_{k+p}\|_g}\\
&\le \|(\nabla - \nabla')|_x\|_{2,g} \sum_j \sup_{W_1,\dots,W_{k+p}}\frac{|\alpha(W_1,\dots,W_{k+p})|}{\|W_1\|_g\dots\|W_{k+p}\|_g}\\
&\le (k+p)\|(\nabla - \nabla')|_x\|_{2,g}\|\alpha\|_g.
\end{align*}
Therefore the $L^2(U,g)$ operator norm of $\nabla - \nabla': C^\infty\Omega^{p,k}(U) \to C^\infty\Omega^{p+1,k}(U)$ is less than or equal to $(k+p)\esssup_{x \in U} \|(\nabla - \nabla')|_x\|_{2,g}$, giving the required bound.

\end{proof}

\section*{Acknowledgments}
The authors were supported by NSF grants DMS-2012427 and DMS-2411208 and the Simons Foundation award MPS-TSM-00002615.

\printbibliography

@dissertation{Li,
       author = "Lizao Li",
       title = "Regge Finite Elements
with Applications in Solid Mechanics and Relativity",
       year = "2018" }

@book{Schwarz,
  title={Hodge Decomposition - A Method for Solving Boundary Value Problems},
  author={Schwarz, G.},
  isbn={9783540494034},
  series={Lecture Notes in Mathematics},
  url={https://books.google.com/books?id=6-17CwAAQBAJ},
  year={2006},
  publisher={Springer Berlin Heidelberg}
}

@article{Holst-Stern,
   title={Geometric variational crimes: {H}ilbert complexes, finite element exterior calculus, and problems on hypersurfaces},
   volume={12},
   ISSN={1615-3383},
   url={http://dx.doi.org/10.1007/s10208-012-9119-7},
   DOI={10.1007/s10208-012-9119-7},
   number={3},
   journal={Foundations of Computational Mathematics},
   publisher={Springer Science and Business Media LLC},
   author={Holst, Michael and Stern, Ari},
   year={2012},
   month=apr, pages={263--293} }

@book{Arnold,
  title={Finite Element Exterior Calculus},
  author={Arnold, D. N.},
  isbn={9781611975536},
  lccn={2018045527},
  series={CBMS-NSF Regional Conference Series in Applied Mathematics },
  url={https://books.google.com/books?id=k7yBDwAAQBAJ},
  year={2018},
  publisher={Society for Industrial and Applied Mathematics}
}

@article{Arnold-Winther,
   title={Finite element exterior calculus: from {H}odge theory to numerical stability},
   volume={47},
   ISSN={1088-9485},
   url={http://dx.doi.org/10.1090/S0273-0979-10-01278-4},
   DOI={10.1090/s0273-0979-10-01278-4},
   number={2},
   journal={Bulletin of the American Mathematical Society},
   publisher={American Mathematical Society (AMS)},
   author={Arnold, Douglas N and Falk, Richard S and Winther, Ragnar},
   year={2010},
   month=jan, pages={281--354} }

@article{GaNe22,
      title={Finite element approximation of scalar curvature in arbitrary dimension}, 
      author={Evan S. Gawlik and Michael Neunteufel},
      journal={Mathematics of Computation},
      year={2025},
      volume={94},
      pages={2685--2722}
}

@article{Licht,
  title={Smoothed projections over manifolds in finite element exterior calculus},
  author={Licht, Martin W},
  journal={arXiv preprint arXiv:2310.14276},
  year={2023}
}

@article{BKGa22,
   title={Finite element approximation of the {L}evi-{C}ivita connection and its curvature in two dimensions},
   volume={24},
   ISSN={1615-3383},
   url={http://dx.doi.org/10.1007/s10208-022-09597-1},
   DOI={10.1007/s10208-022-09597-1},
   number={2},
   journal={Foundations of Computational Mathematics},
   publisher={Springer Science and Business Media LLC},
   author={Berchenko-Kogan, Yakov and Gawlik, Evan S.},
   year={2022},
   month=nov, pages={587--637} }

@book{Darling,
  title={Differential Forms and Connections},
  author={Darling, R.W.R.},
  isbn={9780521468008},
  lccn={93046634},
  url={https://books.google.com/books?id=TdCaahMK0z4C},
  year={1994},
  publisher={Cambridge University Press}
}

@book{Flanders,
  title={Differential Forms with Applications to the Physical Sciences},
  author={Flanders, H.},
  isbn={9780486661698},
  lccn={89036936},
  series={Dover books on advanced mathematics},
  url={https://books.google.com/books?id=pG0PllIO08kC},
  year={1963},
  publisher={Academic Press}
}

@book{Marsden,
  title={Manifolds, Tensor Analysis, and Applications},
  author={Abraham, R. and Marsden, J.E. and Ratiu, T.},
  isbn={9781461210290},
  lccn={89107756},
  series={Applied Mathematical Sciences},
  url={https://books.google.com/books?id=b-IlBQAAQBAJ},
  year={2012},
  publisher={Springer New York}
}

@ARTICLE{Regge,
       author = {{Regge}, T.},
        title = "{General relativity without coordinates}",
      journal = {Il Nuovo Cimento},
         year = 1961,
        month = feb,
       volume = {19},
       number = {3},
        pages = {558-571},
          doi = {10.1007/BF02733251},
       adsurl = {https://ui.adsabs.harvard.edu/abs/1961NCim...19..558R}
}

@book{deRham,
  title={Differentiable Manifolds: Forms, Currents, Harmonic Forms},
  author={Chern, S.S. and Smith, F.R. and de Rham, G.},
  isbn={9783642617522},
  lccn={84010597},
  series={Grundlehren der mathematischen Wissenschaften},
  url={https://books.google.com/books?id=Q1XmCAAAQBAJ},
  year={2012},
  publisher={Springer Berlin Heidelberg}
}

@article{Kanai,
author = {Masahiko Kanai},
title = {{Rough isometries, and combinatorial approximations of geometries of non-compact {R}iemannian manifolds}},
volume = {37},
journal = {Journal of the Mathematical Society of Japan},
number = {3},
publisher = {Mathematical Society of Japan},
pages = {391 -- 413},
year = {1985},
doi = {10.2969/jmsj/03730391},
URL = {https://doi.org/10.2969/jmsj/03730391}
}

@article{Gromov,
author = {Misha Gromov},
title = {{Infinite groups as geometric objects}},
volume = {1},
journal = {Proc. Int. Congress Math. Warsaw},
pages = {385-392},
year = {1984}}

@article{Christiansen-Winther,
author = {Christiansen, Snorre H and Winther, Ragnar},
year = {2008},
month = {04},
pages = {813-829},
title = {Smoothed projections in finite element exterior calculus},
volume = {77},
journal = {Mathematics of Computation},
doi = {10.1090/S0025-5718-07-02081-9}
}

@article{Glinskii,
  title={On the theory of small-amplitude vibrations of curved-surface diaphragms},
  author={Glinskii, GF},
  journal={Technical Physics},
  volume={45},
  pages={8--13},
  year={2000},
  publisher={Springer}
}

@article{Faraudo,
    author = {Faraudo, Jordi},
    title = "{Diffusion equation on curved surfaces. I. Theory and application to biological membranes}",
    journal = {The Journal of Chemical Physics},
    volume = {116},
    number = {13},
    pages = {5831-5841},
    year = {2002},
    month = {04},
    issn = {0021-9606},
    doi = {10.1063/1.1456024},
    url = {https://doi.org/10.1063/1.1456024},
    eprint = {https://pubs.aip.org/aip/jcp/article-pdf/116/13/5831/19308674/5831\_1\_online.pdf},
}

@article{Samavaki,
title = {Navier-{S}tokes equations on {R}iemannian manifolds},
journal = {Journal of Geometry and Physics},
volume = {148},
pages = {103543},
year = {2020},
issn = {0393-0440},
doi = {https://doi.org/10.1016/j.geomphys.2019.103543},
url = {https://www.sciencedirect.com/science/article/pii/S0393044019302244},
author = {Maryam Samavaki and Jukka Tuomela}
}

@article{Dziuk_Elliott_2013,
title={Finite element methods for surface {PDE}s},
volume={22}, DOI={10.1017/S0962492913000056}, 
journal={Acta Numerica}, 
author={Dziuk, Gerhard and Elliott, Charles M.}, 
year={2013},
pages={289--396}}

@article{Calcagni_2013,
   title={Laplacians on discrete and quantum geometries},
   volume={30},
   ISSN={1361-6382},
   url={http://dx.doi.org/10.1088/0264-9381/30/12/125006},
   DOI={10.1088/0264-9381/30/12/125006},
   number={12},
   journal={Classical and Quantum Gravity},
   publisher={IOP Publishing},
   author={Calcagni, Gianluca and Oriti, Daniele and Thürigen, Johannes},
   year={2013},
   month=may, pages={125006} }

@article{demlow2009higher,
  title={Higher-order finite element methods and pointwise error estimates for elliptic problems on surfaces},
  author={Demlow, Alan},
  journal={SIAM Journal on Numerical Analysis},
  volume={47},
  number={2},
  pages={805--827},
  year={2009},
  publisher={SIAM}
}

@article{gawlik2023finite,
  title={Finite element approximation of the {E}instein tensor},
  author={Gawlik, Evan S and Neunteufel, Michael},
  journal={IMA Journal of Numerical Analysis},
  year={2025}
}

@article{gopalakrishnan2023analysis,
  title={Generalizing {R}iemann curvature to {R}egge metrics},
  author={Gopalakrishnan, Jay and Neunteufel, Michael and Sch{\"o}berl, Joachim and Wardetzky, Max},
  journal={arXiv preprint arXiv:2311.01603},
  year={2023}
}

@article{gopalakrishnan2022analysis,
  title={Analysis of curvature approximations via covariant curl and incompatibility for {R}egge metrics},
  author={Gopalakrishnan, Jay and Neunteufel, Michael and Sch{\"o}berl, Joachim and Wardetzky, Max},
  journal={SMAI Journal of Computational Mathematics},
  volume={9},
  pages={151--195},
  year={2023}
}

@article{gopalakrishnan2024improved,
  title={On the improved convergence of lifted distributional {G}auss curvature from {R}egge elements},
  author={Gopalakrishnan, Jay and Neunteufel, Michael and Sch{\"o}berl, Joachim and Wardetzky, Max},
  journal={Results in Applied Mathematics},
  volume={24},
  pages={100511},
  year={2024}
}

@article{gawlik2020high,
  title={High-order approximation of {G}aussian curvature with {R}egge finite elements},
  author={Gawlik, Evan S},
  journal={SIAM Journal on Numerical Analysis},
  volume={58},
  number={3},
  pages={1801--1821},
  year={2020},
  publisher={SIAM}
}

@article{christiansen2024definition,
  title={On the definition of curvature in {R}egge calculus},
  author={Christiansen, Snorre H},
  journal={IMA Journal of Numerical Analysis},
  volume={44},
  issue={5},
  year={2024},
  pages={2698–-2715},
  publisher={Oxford University Press}
}

@article{christiansen2011linearization,
  title={On the linearization of {R}egge calculus},
  author={Christiansen, Snorre H},
  journal={Numerische Mathematik},
  volume={119},
  pages={613--640},
  year={2011},
  publisher={Springer}
}

@article{neunteufel2021avoiding,
  title={Avoiding membrane locking with {R}egge interpolation},
  author={Neunteufel, Michael and Sch{\"o}berl, Joachim},
  journal={Computer Methods in Applied Mechanics and Engineering},
  volume={373},
  pages={113524},
  year={2021},
  publisher={Elsevier}
}

@article{schoeberl1997netgen,
	title={{NETGEN} An advancing front 2{D}/3{D}-mesh generator based on abstract rules},
	author={Sch{\"o}berl, Joachim},
	journal={Computing and Visualization in Science},
	volume={1},
	number={1},
	pages={41--52},
	year={1997},
	publisher={Springer}
}

@article{schoeberl2014ngsolve,
	title={C++ 11 implementation of finite elements in {NGS}olve},
	author={Sch{\"o}berl, Joachim},
	journal={Institute for Analysis and Scientific Computing, Vienna University of Technology},
	year={2014},
	url={https://www.asc.tuwien.ac.at/~schoeberl/wiki/publications/ngs-cpp11.pdf},
}

@article{cheeger1984curvature,
  title={On the curvature of piecewise flat spaces},
  author={Cheeger, Jeff and M{\"u}ller, Werner and Schrader, Robert},
  journal={Communications in Mathematical Physics},
  volume={92},
  number={3},
  pages={405--454},
  year={1984},
  publisher={Springer}
}

@Article{Junya,
AUTHOR = {Takahashi, Junya},
TITLE = {$L^2$-harmonic forms on incomplete {R}iemannian manifolds with positive {R}icci curvature},
JOURNAL = {Mathematics},
VOLUME = {6},
YEAR = {2018},
NUMBER = {5},
ARTICLE-NUMBER = {75},
URL = {https://www.mdpi.com/2227-7390/6/5/75},
ISSN = {2227-7390},
DOI = {10.3390/math6050075}
}

@book{Knapp,
author = {Knapp, Anthony W.},
isbn = {1-4297-9988-9},
title = {Stokes's Theorem and Whitney Manifolds},
year = {2021},
}

@article{gawlik2025curvaturereggemetrics,
      title={On the curvature of {R}egge metrics}, 
      author={Evan S. Gawlik and Jack McKee},
      year={2025},
      journal={arXiv preprint arXiv:2510.25027}
}

@article{kawecki2019finiteelementtheorycurved,
      title={Finite element theory on curved domains with applications to discontinuous {G}alerkin finite element methods}, 
      author={Ellya L. Kawecki},
      year={2020},
      journal={Numerical Methods for Partial Differential Equations},
      volume={36},
      number={6},
      pages={1492--1536}
}

@article{evans2011icesreport1117,
title={Explicit trace inequalities for
isogeometric analysis and parametric
hexahedral finite elements},
author={John A. Evans and Thomas J.R. Hughes},
year={2013},
journal={Numerische Mathematik},
volume={123},
pages={259-–290}
}

@article{bernardi1989optimal,
  title={Optimal finite-element interpolation on curved domains},
  author={Bernardi, Christine},
  journal={SIAM Journal on Numerical Analysis},
  volume={26},
  number={5},
  pages={1212--1240},
  year={1989},
  publisher={SIAM}
}

@article{christiansen2007stability,
  title={Stability of {H}odge decompositions in finite element spaces of differential forms in arbitrary dimension},
  author={Christiansen, Snorre H},
  journal={Numerische Mathematik},
  volume={107},
  number={1},
  pages={87--106},
  year={2007},
  publisher={Springer}
}

\end{document}